\documentclass[a4paper,11pt]{article}
\usepackage[utf8]{inputenc}
\usepackage[english]{babel}
\usepackage{graphicx}
\usepackage{fancyhdr}
\usepackage{geometry}
\usepackage{amsfonts,amssymb,amsthm,amsmath}
\usepackage{mathtools}
\usepackage{enumitem}
\usepackage[percent]{overpic}
\usepackage{tikz}
\usepackage{mathrsfs}
\usepackage{ifthen,tabularx,graphicx,multirow}
\usepackage{xargs}
\usepackage[colorinlistoftodos,prependcaption,textsize=tiny]{todonotes}
\usepackage{tcolorbox}
\usepackage{enumerate}
\usepackage{comment}
\usepackage{algorithm}
\usepackage{arydshln}
\usepackage[
	colorlinks,
	linkcolor={blue},
	citecolor={black},
	urlcolor={black}
]{hyperref}
\usepackage{algpseudocode}
\usepackage{booktabs}
\usepackage{varwidth}
\usepackage[capitalise]{cleveref}
\usepackage{cite}
\usepackage{changepage}
\usepackage{caption}

\newtheorem{theorem}{Theorem}
\newtheorem{lemma}{Lemma}
\newtheorem{definition}{Definition}
\crefformat{equation}{(#2#1#3)}
\newcommand*{\dd}{{\,\mathrm{d}}}

\makeatletter
\def\@email#1#2{%
 \endgroup
 \patchcmd{\titleblock@produce}
  {\frontmatter@RRAPformat}
  {\frontmatter@RRAPformat{\produce@RRAP{*#1\href{mailto:#2}{#2}}}\frontmatter@RRAPformat}
  {}{}
}%
\makeatother

\providecommand{\keywords}[1]{\textbf{\textit{Keywords:}} #1}

\title{Computing Generalized Eigenfunctions\\
in Rigged Hilbert Spaces}
\author{Matthew J. Colbrook\thanks{Department of Applied Mathematics and Theoretical Physics, University of Cambridge, Cambridge, UK (m.colbrook@damtp.cam.ac.uk)}
\and Andrew Horning\thanks{Department of Mathematical Sciences, Rensselaer Polytechnic Institute, Troy, NY, USA}
\and Tianyiwa Xie\thanks{Department of Applied Mathematics and Theoretical Physics, University of Cambridge, Cambridge, UK}}

\date{}

\begin{document}

\maketitle
\begin{abstract}
We introduce a simple, general, and convergent scheme to compute generalized eigenfunctions of self-adjoint operators with continuous spectra on rigged Hilbert spaces. Our approach does not require prior knowledge about the eigenfunctions, such as asymptotics or other analytic properties. Instead, we carefully sample the range of the resolvent operator to construct smooth and accurate wave packet approximations to generalized eigenfunctions. We prove high-order convergence in key topologies, including weak-$^*$ convergence for distributional eigenfunctions, uniform convergence on compact sets for locally smooth generalized eigenfunctions, and convergence in seminorms for separable Frechet spaces, covering the majority of physical scenarios. The method's performance is illustrated with applications to both differential and integral operators, culminating in the computation of spectral measures and generalized eigenfunctions for an operator associated with Poincaré’s internal waves problem. These computations corroborate experimental results and highlight the method's utility for a broad range of spectral problems in physics.
\end{abstract}



\section{Introduction}

The spectrum of any finite matrix consists only of discrete eigenvalues. However, many linear operators in physics have a \textit{continuous} spectral component. Common examples include differential operators \cite{izrail1965direct,levitan1975introduction,titchmarsh2011elgenfunction}, coupled systems of differential equations \cite{grubb1977essential,kako1987essential,raikov1991spectrum,langer1996essential}, as well as differential-integral operators and singular integral operators \cite{friedrichs1948perturbation,koppelman1960spectral,cuminato2007review,muskhelishvili2008singular}. In particular, operators associated with physical systems that scatter or radiate energy often exhibit a combination of continuous and discrete spectra \cite{amrein1973characterization,enss1978asymptotic,ruelle1969remark}.

If a self-adjoint operator $A$ has a continuous spectral component, then the eigenvectors of $A$ do not form a basis for the Hilbert space $H$ and do not diagonalize $A$. However, one may look beyond $H$ for \textit{generalized eigenfunctions}, linear functionals acting on dense subspaces of $H$, associated with points in the continuous spectrum. Generalized eigenfunction expansions have a rich history with roots in the theory of singular Sturm--Liouville problems~\cite{weyl1910gewohnliche}, ordinary differential equations~\cite{krein1946general,krein1948hermitian,krein1949fundamental,krein1950one,gelfand1955eigenfunction},  and elliptic partial differential equations~\cite{browder1954eigenfunction,gaarding1954application,mautner1953eigenfunction,berezansky1956generalization,berezanskii1957eigenfunction,kac1961spectral}. See for example, the work of Povzner \cite{povzner1953expansion}, Ikebe \cite{ikebe1960eigenfunction}, and Agmon~\cite{agmon1975spectral}, who used these expansions in scattering theory, and the abstract generalizations given by Kuroda \cite{kuroda1967perturbation,kuroda1967abstract} and Kato and Kuroda \cite{kato1970theory}. Later, Gelfand and collaborators~\cite{gel1967generalized,gel2016generalized} showed that generalized eigenfunction expansions for a broad class of self-adjoint operators can be established within a \textit{rigged Hilbert space}, that is, a Gelfand triple consisting of three nested spaces $\Phi \subset H\subset \Phi^*$.\footnote{The term ``rigged" is a translation of the Russian participle ``osnashchyonoe" indicating that the Hilbert space has been equipped with the additional structure of the triple $\Phi\subset H\subset \Phi^*$.} Here, $\Phi$ is a nuclear space and $\Phi^*$ is its dual \cite[Section 29.1]{blanchard2015mathematical}. The nuclear spectral theorem states that if $A$ is self-adjoint with $A\Phi\subset \Phi$, then $A$ has a complete set of generalized eigenfunctions in $\Phi^*$ \cite{gel1967generalized,gel2016generalized}. It is often stated that employing the rigged Hilbert space framework for diagonalization affords a physically meaningful and mathematically rigorous interpretation of the Dirac bra-ket formulation commonly utilized in quantum mechanics~\cite{roberts1966rigged,de2005role}.

Despite the appearance of continuous spectra in domains across pure and applied mathematics, there have been relatively few general and convergent numerical algorithms to compute continuous spectral properties. Recently, such schemes been developed for projection-valued spectral measures \cite{colbrook2019computing,colbrook2020computing,2023computingTI}. However, several key challenges stand in the way of a computational framework for generalized eigenfunction expansions.
\begin{itemize}[leftmargin=*,noitemsep]
\item Generalized eigenfunctions are linear functionals on an infinite-dimensional space. \textit{Given finite resources on a computer, what form should our approximations take?}
\item Theoretical foundations for computing spectral measures are inherently Hilbert space theoretic, yet nontrivial generalized eigenfunctions are not elements of a Hilbert space. \textit{Can we develop a rigorous foundation for our computational framework?}
\item In many practical situations, generalized eigenfunctions have important additional regularity beyond that of a generic linear functional (for example, locally smooth generalized eigenfunctions of singular Sturm--Liouville problems). \textit{Can our general approximation scheme capture these additional fine properties?}
\end{itemize}
To address these questions, we outline a simple procedure for computing high-order approximations to generalized eigenfunctions in a rigged Hilbert space. By carefully sampling the resolvent operator in the complex plane, we construct smooth wave packet approximations formed from a narrow band of spectral content. These wave packets are easy to work with on the computer because they are proper functions in a smooth subspace of the Hilbert space. However, as the effective bandwidth of spectral content decreases, these wave packets approximate the action of generalized eigenfunctions within the rigged Hilbert space. Our approach can be understood as a generalization of Stone's formula~\cite{Stone} and the limiting absorption principle to high-order approximations, offering improved accuracy and stability for numerical computation.

Our approximations are accompanied by a rigorous theory of high-order convergence. We show that the wave packets converge to generalized eigenfunctions in the weak-$^*$ topology on $\Phi^*$, i.e., that the coefficients of test functions in the generalized eigenfunction expansion converge at a rate controlled by the local regularity of the spectral measure (see~\cref{thm:conv_ae,thm:conv_rates}). Moreover, we show that the wave packet approximations capture important and natural fine properties of the generalized eigenfunctions: they converge uniformly on compact sets where the generalized eigenfunctions can be identified locally with continuous functions (see~\cref{thm:conv_point}). We also provide general conditions for weak and seminorm convergence to generalized eigenfunctions in Suslin spaces (see~\cref{thm:conv_abs}), a class of separable locally convex spaces that capture most  physically relevant spaces used in applied analysis, including Frechet spaces (hence, Banach and Hilbert spaces) and the usual spaces of distributions and test functions~\cite{thomas1975integration}.

Crucially, our method is widely applicable and easy to implement. It does not require prior knowledge of the eigenfunctions, such as the matching asymptotics typically used in scattering algorithms. The result is a theoretically sound, user-friendly tool that opens the door to computational spectral analysis in rigged Hilbert spaces to mathematicians and scientists alike. The source code for the algorthms and numerical experiments in this paper are available in a public repository at \href{https://github.com/SpecSolve/GenEigs}{https://github.com/SpecSolve/GenEigs}.

The paper is organized as follows. In \cref{sec:diagonalization}, we discuss generalized eigenfunction expansions and their construction via spectral measures. We illustrate with two canonical examples. \cref{sec:comput_geneigs} introduces our method and several convergence theorems. We include a variety of numerical examples with differential and integral operators. In \cref{sec:int_waves}, we use our method to compute spectral measures and generalized eigenfunctions of an operator that governs Poincar\'e's internal waves problem. Recent interest in the analysis of this operator and its spectral properties has surged~\cite{de2020attractors,colin2020spectral,dyatlov2019microlocal,tao20190,wang2022dynamics,wangscattering,dyatlov2021mathematics}. We demonstrate the effectiveness of our methods by computing generalized eigenfunctions, which corroborate experimental realizations of internal waves as observed in \cite{hazewinkel2010observations}.

\vspace{2mm}

\noindent{}\textbf{Notation:} We denote the duality pairing between a topological vector space $\Phi$ and its dual $\Phi^*$ by $\langle\cdot | \cdot\rangle$. The inner product on a Hilbert space $H$ is $\langle\cdot,\cdot\rangle$, which is conjugate linear in the second argument, and the associated Hilbert norm is $\|\cdot\|$. Throughout, $A:D(A)\rightarrow H$ is a self-adjoint operator on $H$ with domain $D(A)$ and spectrum denoted by the set $\Lambda(A)\subset\mathbb{R}$. The resolvent of $A$ is $R_A(z)=(A-z)^{-1}$, a bounded operator on $H$ for every $z\not\in\Lambda(A)$ with operator norm $\|R_A(z)\|=1/\mathrm{dist}(z,\Lambda(A))\leq 1/|{\rm Im}(z)|$. The characteristic function of a set $\Omega$ is denoted by $\chi_\Omega$.

\section{Diagonalization in a rigged Hilbert space}\label{sec:diagonalization}

Here, we examine two classical methods for diagonalizing a self-adjoint operator on a separable Hilbert space $H$. The first method, suitable for any self-adjoint operator, employs projection-valued spectral measures. The second method is applicable specifically when the Hilbert space is equipped with a dense nuclear subspace; here, generalized eigenfunctions are employed to diagonalize the operator. These generalized eigenfunctions can be explicitly constructed using the projection-valued measure.

\subsection{Diagonalization by spectral measures}\label{sec:measures}

A self-adjoint operator $A$ acting on a separable Hilbert space $H$ may not possess a basis of eigenfunctions, or indeed any eigenfunctions at all. Nevertheless, it is always diagonalizable using a \textit{projection-valued measure}, as described in \cite[Thm.~VIII.6]{reed1972methods}. A projection-valued measure, also known as a resolution of the identity, is a countably additive function $E:\mathcal{B}(\Lambda(A))\rightarrow B(H)$ that assigns an orthogonal projector to each Borel-measurable set $S\in\mathcal{B}(\Lambda(A))$. These orthogonal projections are complete in $H$ and diagonalize $A$ in the sense that
\begin{equation}\label{eqn:cont_decomp}
	f=\int_\mathbb{R} \dd E(\lambda)f\quad \forall f\in H \qquad\text{and}\qquad Au=\int_\mathbb{R} \lambda\dd E(\lambda)u \quad \forall u\in D(A).
\end{equation}
The projection-valued measure $E$ extends the notion of orthogonal projectors onto eigenspaces, typical for a self-adjoint matrix, to the broader context of self-adjoint operators. The primary goal of our study is to extend this framework to compute projections onto \textit{generalized eigenspaces} for operators, especially those with a continuous spectrum.

Given any $f\in H$, the scalar-valued spectral measure of $A$ with respect to $f$ is defined by $\mu_f(S)=\langle E(S)f,f\rangle$. If $f$ is normalized such that $\|f\|=1$, then $\mu_f$ is a probability measure on $\mathbb{R}$. In quantum mechanics, where $A$ represents an observable and $f$ a quantum state, $\mu_f$ describes the likelihood of various outcomes upon measuring the observable. The Lebesgue decomposition of $\mu_f$~\cite{stein2009real} allows us to write the spectral measure as a sum of absolutely continuous, singularly continuous, and discrete components, where we take absolutely continuous with respect to the Lebesgue measure $\dd\lambda$. This can be represented as:
\begin{equation}\label{eqn:spec_meas}
\dd\mu_f(\lambda)= \underbrace{\rho_f(\lambda)\dd\lambda +\dd\mu_f^{(\mathrm{sc})}(\lambda)}_{\text{continuous part}} + \underbrace{\sum_{\lambda_*\in\Lambda^{{\rm p}}(A)}\langle P_{\lambda_*} f,f\rangle\,\delta({\lambda-\lambda_*})\dd\lambda}_{\text{discrete part}}.
\end{equation}
Here, $\Lambda^{{\rm p}}(A)$ is the set of eigenvalues of $A$, $\delta(\lambda-\lambda_*)$ denotes a Dirac point mass located at eigenvalue $\lambda_*$ of $A$ and $P_{\lambda_*}$ is the orthogonal projector onto the invariant subspace associated with $\lambda_*$. The mixed spectral measure with respect to $f_1,f_2\in H$ is defined by $\mu_{f_1,f_2}(S):=\langle E(S)f_1,f_2\rangle$ for $S\in\mathcal{B}(\Lambda(A))$, and can be recovered from the set of measures $\{\mu_f:f\in H\}$ using the polarization identity.

\subsection{Diagonalization by generalized eigenfunctions}\label{sec:gefs}

While the eigenfunctions of $A$ may not span $H$, many operators encountered in applications possess \textit{generalized eigenfunctions}. These are distributions that act on a dense subset of \( H \). Suppose \( \Phi \) is a topological vector space that embeds continuously and densely into \( H \). Let \( \Phi^* \) denote its topological dual, which is the space of continuous linear functionals on \( \Phi \). Under the typical identification \( H = H^* \), the embedding \( H \hookrightarrow \Phi^* \) is continuous and dense. Moreover, the inner product on \( H \) aligns with the duality pairing between \( \Phi \) and \( \Phi^* \). Specifically, for \( \psi \in H \) (also in \( \Phi^* \)) and \( \phi \in \Phi \) (also in \( H \)), the relation \( \langle\psi|\phi\rangle = \langle\psi, \phi\rangle \) holds. The triplet of embedded spaces, \( \Phi \subset H \subset \Phi^* \), is termed a \textit{rigged Hilbert space}. This is also called a Gel'fand triple.

If the space \( \Phi \) is invariant under \( A \), meaning \( A\Phi \subset \Phi \), we can search for distributional solutions to the eigenvalue problem. A functional \( \psi \in \Phi^* \) is termed a \textit{generalized eigenfunction} of \( A \) if there exists a scalar \( \lambda \) such that
\begin{equation}\label{eqn:def_geneig}
\langle \psi | A\phi \rangle = \lambda \langle \psi | \phi \rangle \qquad \forall\phi \in \Phi.
\end{equation}
Notably, any eigenfunction \( u \in H \) is also a generalized eigenfunction. This is because \( \langle u | A\phi \rangle = \langle u, A\phi \rangle = \langle Au, \phi \rangle = \lambda \langle u, \phi \rangle = \lambda \langle u | \phi \rangle \). If $\Phi$ is a \textit{nuclear space}, the generalized eigenfunctions are complete and diagonalize \( A \). The space \( \Phi \) is termed \textit{countably Hilbert} if its topology is generated by a countable system of compatible norms, which are induced by inner products~\cite[Ch.~3.1]{gel2016generalized}. Moreover, it is called nuclear if every one-parameter family of continuous linear functionals on $\Phi$ with weakly bounded variation also has strongly bounded variation (see, e.g., p.~$178$ of Gel'fand and Shilov\cite{gel1967generalized} for further construction, characterization, and illustrations of nuclear spaces).

When \( \Phi \) is a nuclear space, a complete set of generalized eigenfunctions diagonalizing \( A \) can be constructed explicitly using the projection-valued measure \( E \) as follows~\cite[Ch.~IV]{gel1967generalized}. For $a<b$, we let $\Delta_a^b=(a,b)$. Consider a fixed $f\in H$. For any $h>0$, note that $\smash{\|E(\Delta_\lambda^{\lambda+h})f\|^2=\mu_f(\Delta_\lambda^{\lambda+h})}$. Suppose that $\smash{\mu_f(\Delta_\lambda^{\lambda+h})>0}$, then we can define the vector
$$
u_{\lambda}^{(h)}=\frac{1}{\mu_f(\Delta_\lambda^{\lambda+h})}\cdot E(\Delta_\lambda^{\lambda+h})f=\frac{1}{\mu_f(\Delta_\lambda^{\lambda+h})}\cdot \int_{\Delta_\lambda^{\lambda+h}} 1 \dd E(y)f.
$$
It follows that
$$
\left\|A u_{\lambda}^{(h)}-\lambda u_{\lambda}^{(h)}\right\|=\left\|\int_{\Delta_\lambda^{\lambda+h}} (y-\lambda) \dd E(y) u_{\lambda}^{(h)}\right\| \leq h\left\|u_{\lambda}^{(h)}\right\|,
$$
so that $u_{\lambda}^{(h)}$ is `almost' an eigenfunction. This motivates taking the limit $h\downarrow 0$ in $\Phi^*$, and it can be shown that the weak limit
\begin{equation}\label{eqn:pvm_to_geig}
u_\lambda = \lim_{h\downarrow 0}\frac{1}{\mu_f(\Delta_\lambda^{\lambda+h})}\cdot E(\Delta_\lambda^{\lambda+h})f,
\end{equation}
exists for $\mu_f$-a.e. $\lambda\in\mathbb{R}$ as a continuous linear functional on $\Phi$. Moreover, the measure $\mu_{f,\phi}$ is $\mu_f$-absolutely continuous for all $\phi\in\Phi$ and the action of $u_\lambda$ is given by the Radon-Nikodym derivative~\cite[pp.~184--18]{gel1967generalized}
\begin{equation}\label{eqn:geig_action}
\langle u_\lambda| \phi\rangle = \frac{\mathrm{d}\mu_{f,\phi}}{\mathrm{d}\mu_f}(\lambda)\qquad\forall\phi\in\Phi.
\end{equation}
The functional $u_\lambda\in\Phi^*$ is a generalized eigenfunction of $A$, satisfying $\langle u_\lambda|A\phi\rangle = \lambda\langle u_\lambda|\phi\rangle$ for all $\phi\in\Phi$. Moreover, any function $\phi\in \Phi$ in the subspace of $H$ spanned by functions of the form $E(\Delta_{-\infty}^\lambda)f$ has the expansion
\begin{equation}\label{eqn:completeness}
\phi = \int_\mathbb{R} \overline{\langle u_\lambda|\phi\rangle} u_\lambda\dd\mu_f(\lambda), \qquad\text{with}\qquad \|\phi\|^2=\int_\mathbb{R} |\langle u_\lambda |\phi \rangle|^2\dd\mu_f(\lambda).
\end{equation}
Since $H$ can be written as a direct sum of such subspaces, one can construct a complete set of generalized eigenfunctions of $A$ that have the form in~\eqref{eqn:pvm_to_geig}. That is, let $H=\oplus_{k\in S} H_k$ where each subspace $H_k$ is spanned by $\{E(\Delta_{-\infty}^\lambda)f_k\}_{\lambda\in\mathbb{R}}$, for  $f_k\in H$, $k\in S$, where $S$ is countable. There is a complete set of generalized eigenfunctions $u_{\lambda,k}\in\Phi^*$ such that, for any $\phi\in\Phi$,
\begin{equation}
\phi = \sum_{k\in S}\int_\mathbb{R} \overline{\langle u_{\lambda,k}|\phi\rangle} u_{\lambda,k}\dd\mu_{f_k}(\lambda), \qquad\text{and}\qquad A\phi = \sum_{k\in S}\int_\mathbb{R} \lambda\overline{\langle u_{\lambda,k}|\phi\rangle} u_{\lambda,k}\dd\mu_{f_k}(\lambda).
\end{equation}
This expansion can be found in, e.g., Remark~$1$ on pp.~$188$-$189$ of Gel'fand and Shilov\cite{gel1967generalized}.
The number of nonzero functionals $u_{\lambda,k}$, for $k\in S$, associated with a point $\lambda\in\Lambda(A)$ is the \textit{multiplicity} of the generalized eigenvalue $\lambda$. 

\subsection{Two canonical examples}\label{sec:canon_ex}

To illustrate the abstract framework for diagonalization in a concrete setting, we introduce two canonical examples. We will return to these examples to illustrate fundamental approximation properties and numerical considerations for our scheme in~\cref{sec:comput_geneigs}.

\subsubsection{Multiplication operator}\label{sec:mult_ex}

First, let $\Omega = (-1, 1)$ and consider the multiplication operator $P$ on $H=L^2(\Omega)$ defined by $[Pu](x) = p(x)u(x)$, where $p(x)=x^3-x$. We rig $H$ with the nuclear space $\Phi=C^\infty(\Omega)$ of infinitely differentiable functions on $\Omega$ and its continuous dual, the space of compactly supported distributions on $\Omega$. Polynomial multiplication operators and perturbations thereof play a distinguished role in the theory of differential and pseudodifferential operators and provide simple pedagogical illustrations of key phenomena~\cite{halmos1982hilbert,arveson2002short,weidmann2012linear}. 

\begin{figure}[tbp!]
    \centering
    \begin{minipage}{0.48\textwidth}
        \begin{overpic}[width=\textwidth]{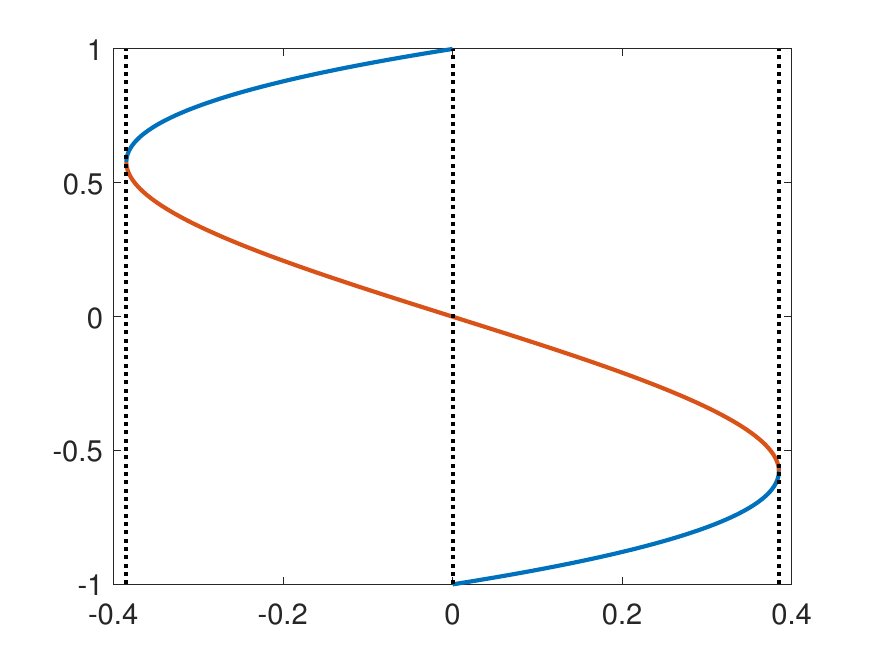}
        \put (50,-2) {$\displaystyle \lambda$}
        \put (12,72) {$\displaystyle \text{Inverse branches of } p(x)=x^3-x $}
        \put (30,61) {{$\displaystyle p_1^{-1}(\lambda)$}}
        \put (38,45) {{$\displaystyle p_2^{-1}(\lambda)$}}
        \put (62,16) {{$\displaystyle p_1^{-1}(\lambda)$}}
        \end{overpic}
    \end{minipage}
    \begin{minipage}{0.48\textwidth}
        \begin{overpic}[width=\textwidth]{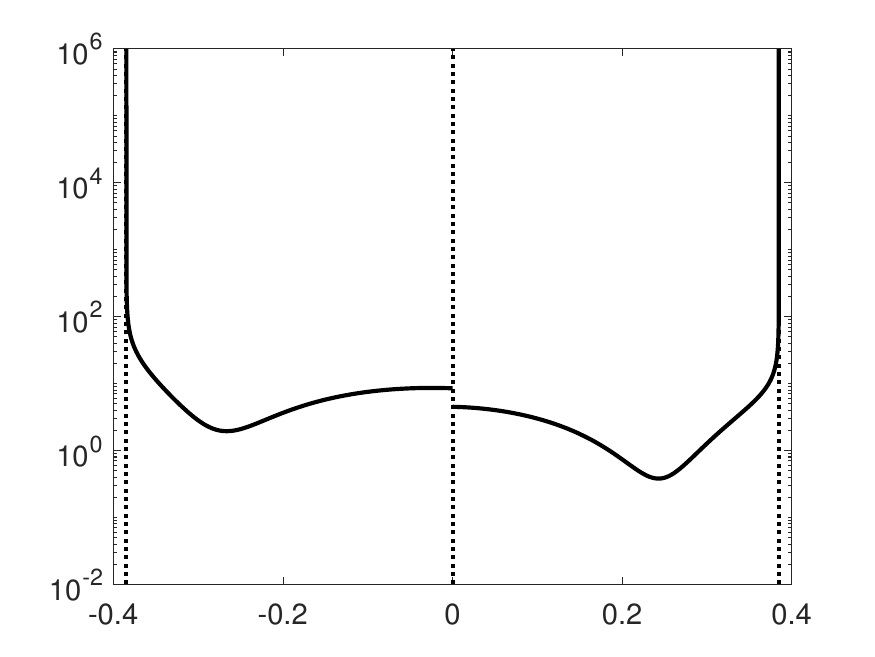}
        \put (50,-2) {$\displaystyle \lambda$}
        \put (48,72) {$\displaystyle \rho_{f}(\lambda) $}
        \end{overpic}
    \end{minipage}
    \caption{\label{fig:mult_op} Branches of the inverse of $p(x)=x^3-x$, denoted $p_k^{-1}$ for $k=1,2$, are plotted in the left panel. The associated spectral measure of the multiplication operator $[Pu](x)=(x^3-x)u(x)$ with respect to $f(x)=(2+x)\cos(2\pi x)$ has singularities at the endpoints and the origin.}
\end{figure}
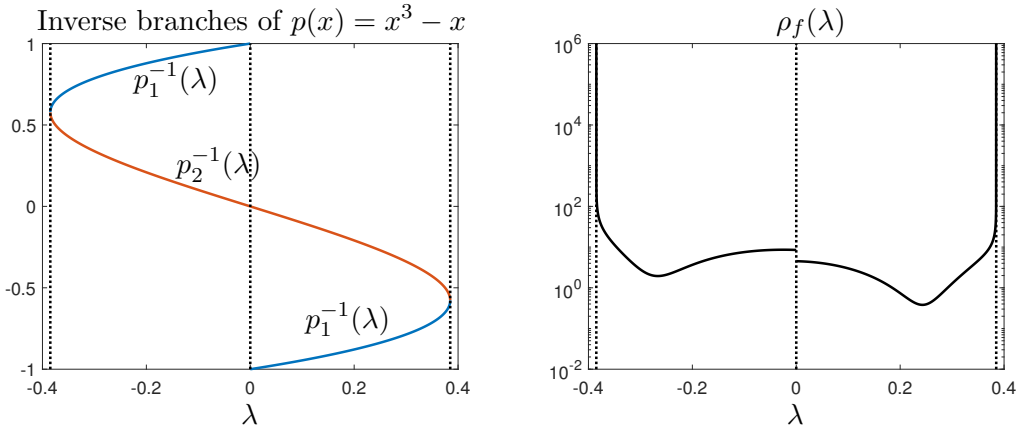

The spectrum of $P$ is absolutely continuous with $\Lambda(P)=[-2\sqrt{3}/9,\,\, 2\sqrt{3}/9]$, which is simply the closure of the range of $p(x)$ on $\Omega$. Each point $\lambda$ in the spectrum is associated with generalized eigenfunctions of the form $\delta(x-x_k)$ where $p(x_k)=\lambda$. In other words, the generalized eigenfunctions are Dirac delta distributions centered at points in the preimage of $\lambda$ under the map $x\rightarrow p(x)$. The spectral measure $\mathrm{d}\mu_{f,g}=\rho_{f,g}\dd\lambda$ can be inferred from the quadratic form via a change of variables $\lambda=p(x)$:
\begin{equation}\label{eqn:multop_spectral_measure1}
\begin{aligned}
\langle Pf,g\rangle = \int_{-1}^1 p(x)f(x)g(x)\dd x = \int_{\Lambda(P)}
\end{aligned} \lambda \,\sum_{k=1}^2\left[\frac{f(p_k^{-1}(\lambda))g(p_k^{-1}(\lambda))}{p'(p_k^{-1}(\lambda))}\right]\dd\lambda.
\end{equation}
Here, each function $p^{-1}_k$, $k=1,2$, is a branch of the inverse of $p(x)$ with domain $\Lambda(P)$. Comparing~\cref{eqn:multop_spectral_measure1} with the diagonalization identity in~\cref{eqn:cont_decomp} (after taking inner products with $g$ on both sides) leads to the identification (a.e. with respect to Lebesgue measure)
\begin{equation}\label{eqn:multop_spectral_measure2}
\rho_{f,g}(\lambda) = \sum_{k=1}^2 \left[\frac{f(p_k^{-1}(\lambda))g(p_k^{-1}(\lambda))}{p'(p_k^{-1}(\lambda))}\right] \qquad\text{for almost every} \qquad \lambda\in\Lambda(P).
\end{equation}
The branches $p^{-1}(k)$ and the Radon-Nikodym $\rho_f$ of the spectral measure with respect to $f(x) = (2+x)\cos(2\pi x)$ are plotted in~\cref{fig:mult_op} with the singularity locations of $\rho_f$ indicated by dashed lines in each. Notice that the spectrum has multiplicity $2$, corresponding to the two branches of the inverse, except at the point $\lambda=0$ where the branch $p^{-1}(\lambda)$ has a jump discontinuity and the spectrum has multiplicity $1$. When $f,g\in\Phi$, the density $\smash{\rho_{f,g}}$ is infinitely differentiable except for a possible jump discontinuity at $\lambda=0$ and blow-up singularities at the endpoints $\lambda = \pm 2\sqrt{3}/9$ of the spectrum. We examine several examples in detail in~\cref{sec:comput_geneigs} to illustrate the connection between singularities of spectral measures and multiplicity changes, and we quantify the local influence of singularities in the Radon-Nikodym $\rho_f$ on our numerical approximations in~\cref{sec:convergence}.

\subsubsection{Differential operator}\label{sec:diff_ex}

Next, consider the differential operator $Au=-u''$ defined on $W^2(\mathbb{R})$, the space of square-integrable functions on the real line possessing two weak square-integrable derivatives. Then, $A:W^2(\mathbb{R})\rightarrow L^2(\mathbb{R})$ and we can rig $L^2(\mathbb{R})$ with the dense subspace of functions in the Schwartz space $S(\mathbb{R})$ and its dual $S^*(\mathbb{R})$, the space of tempered distributions. The spectrum of $A$ is absolutely continuous on the positive real axis and $A$ is unitarily equivalent to multiplication by $\smash{x^2}$ under the Fourier transform.

Although $A$ has no eigenfunctions in $L^2(\mathbb{R})$, we can calculate its generalized eigenfunctions with two integration-by-parts:
$$
\int_{-\infty}^\infty e^{\pm 2\pi ikx}A\phi(x)\dd x = 4\pi^2k^2\int_{-\infty}^\infty e^{\pm 2\pi ikx}\phi(x)\dd x,
$$
whenever $\phi\in S(\mathbb{R})$. This means that the generalized eigenfunctions of $A$ are Fourier modes, which are complete in $S(\mathbb{R})$. Note that there are two generalized eigenfunctions $\exp(\pm i\sqrt{\lambda} x)$ associated with each point $\lambda\in\Lambda(A)=[0,+\infty)$, that is, each generalized eigenvalue has multiplicity two. With a change-of-variables argument in the Fourier domain, similar to~\cref{eqn:multop_spectral_measure1}, we obtain
$$
\rho_{f,g}(\lambda) = \frac{1}{4\pi\sqrt{\lambda}}\left[\hat f\left(\frac{\sqrt{\lambda}}{2\pi}\right)\hat g\left(\frac{\sqrt{\lambda}}{2\pi}\right)+\hat f\left(\frac{-\sqrt{\lambda}}{2\pi}\right)\hat g\left(\frac{-\sqrt{\lambda}}{2\pi}\right)\right],
$$
where $\hat u(k) = \int_\mathbb{R} e^{-2\pi i kx} u(x)\,\mathrm{d}x$ denotes the Fourier transform of the square integrable function $\smash{u\in L^2(\mathbb{R})}$. If $f,g\in\Phi$, then the spectral measure is again infinitely differentiable except possibly at the endpoint $\lambda=0$ of the spectrum. The inverse branches of the Fourier multiplier $\smash{p(x)=x^2}$ associated with $A$ are plotted in~\cref{fig:diff_op} (left panel) along with the Radon-Nikodym derivative $\rho_f$ of the spectral measure with respect to $\exp(-\pi x^2)$ (right panel). The singularity of $\rho_f$ at $\lambda=0$ is indicated by a dashed line.

\begin{figure}[tbp!]
    \centering
    \begin{minipage}{0.48\textwidth}
        \begin{overpic}[width=\textwidth]{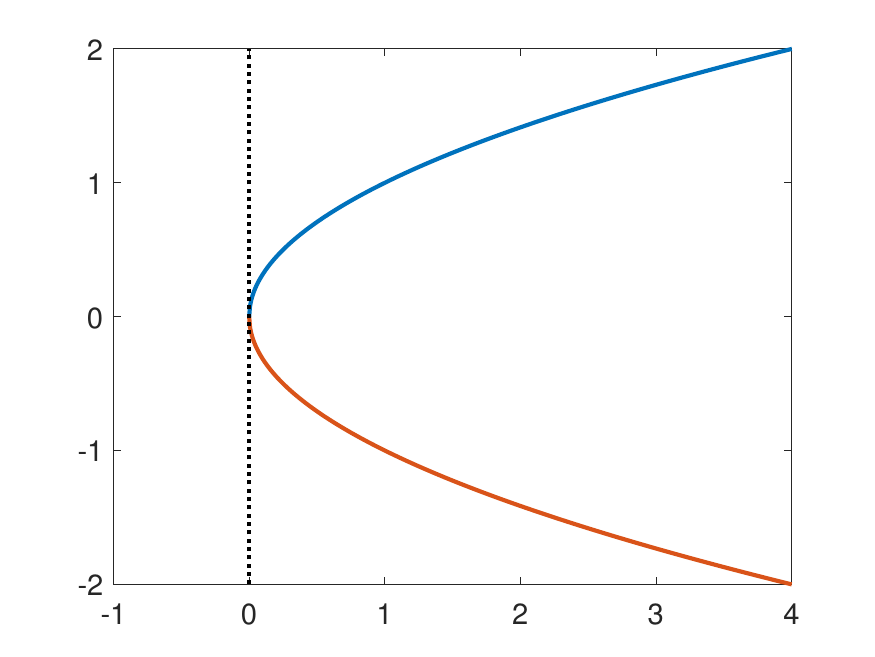}
        \put (50,-2) {$\displaystyle \lambda$}
        \put (17,72) {$\displaystyle \text{Inverse branches of } p(x)=x^2 $}
        \put (50,52) {{$\displaystyle p_1^{-1}(\lambda)=\sqrt{\lambda}$}}
        \put (50,24) {{$\displaystyle p_2^{-1}(\lambda)=-\sqrt{\lambda}$}}
        \end{overpic}
    \end{minipage}
    \begin{minipage}{0.48\textwidth}
        \begin{overpic}[width=\textwidth]{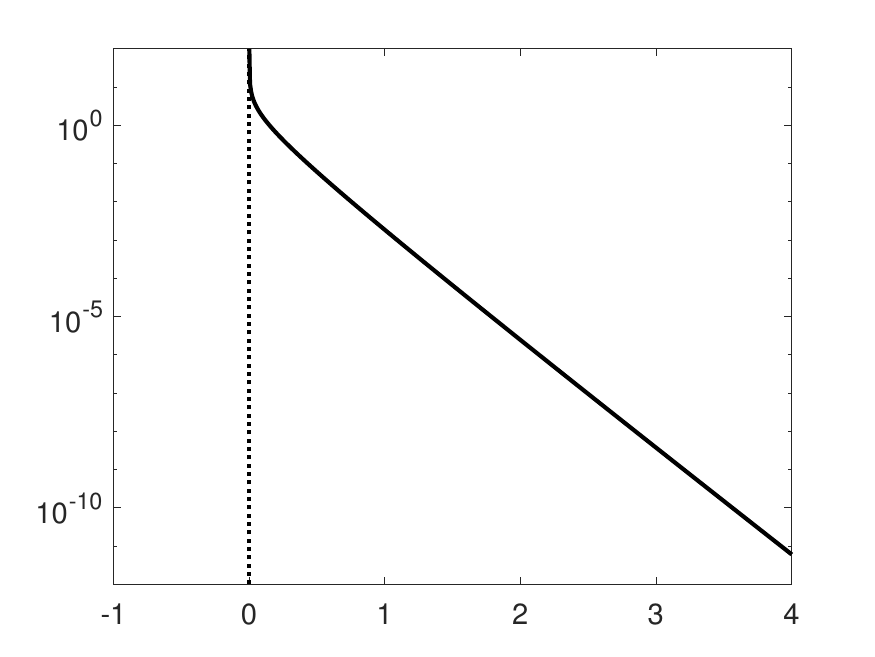}
        \put (50,-2) {$\displaystyle \lambda$}
        \put (48,72) {$\displaystyle \rho_{f}(\lambda) $}
        \end{overpic}
    \end{minipage}
    \caption{\label{fig:diff_op} Branches of the inverse of the Fourier multiplier $p(x)=x^2$ associated with the operator $[Au](x)=-u''(x)$. The associated spectral measure of $A$ with respect to $f(x)=\exp(-\pi x^2)$ has a single singularity at the endpoint $\lambda=0$ of the spectrum.}
\end{figure}
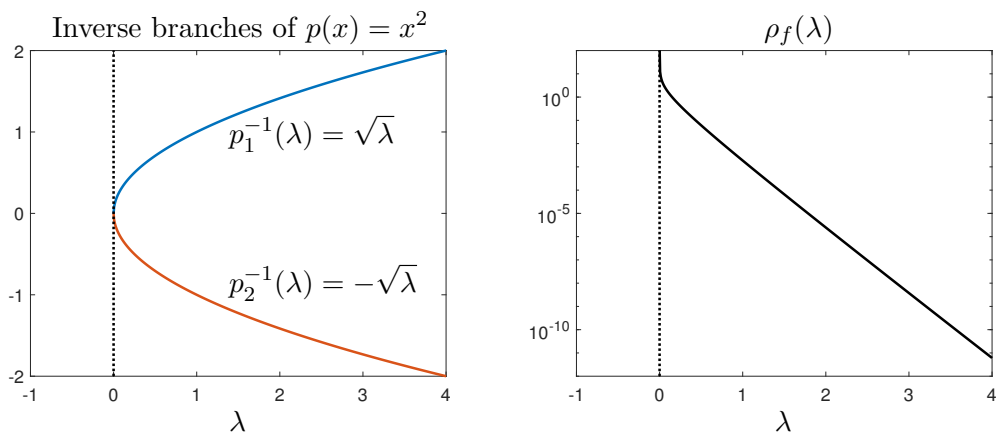

\section{Computing generalized eigenfunctions}\label{sec:comput_geneigs}

To compute generalized eigenfunctions, we employ \textit{wave packet} approximations of the form:
\begin{equation}\label{eqn:wave_packet}
u_\lambda^{(\epsilon)} = \int_\mathbb{R} K_\epsilon(\lambda - \tilde\lambda)\dd E(\tilde\lambda)f, \qquad f\in H.
\end{equation}
In this expression, $f$ is a fixed vector in $H$.
The family $\{K_\epsilon\}_{\epsilon>0}$ represents approximate identities, also known as smoothing kernels, with $K_\epsilon(x)=K(x/\epsilon)/\epsilon$ for some $K\in L^1(\mathbb{R})$. As $\epsilon\rightarrow 0$, the approximation $\smash{u^{(\epsilon)}_\lambda\in H}$ is formed from a wave packet of generalized eigenfunctions increasingly concentrated near the point $\lambda$. One could also consider the choice of $K_\epsilon$ as $\smash{\chi_{[\lambda,\lambda+\epsilon]}/\mu_f(\Delta_\lambda^{\lambda+\epsilon})}$, which recovers the difference quotient in~\cref{eqn:def_geneig}. However, this choice is computationally impractical. A computational scheme relying on this particular $K_\epsilon$ necessitates further approximations to evaluate the convolution in~\cref{eqn:wave_packet}. As an alternative, we utilize a family of \textit{rational kernels} that approximate the identity, allowing for a direct evaluation of the convolution in~\cref{eqn:wave_packet} using the resolvent $R_A(z)$. Our approach generalizes schemes based on Stone's theorem and limiting absorption principles to high-order approximations with improved accuracy and stability properties for numerical computation.

We begin with a basic illustration using the Poisson kernel to elucidate key approximation properties in the context of two canonical examples: multiplication and differentiation operators. However, the Poisson kernel leads to a method with a slow, order-one, convergence rate. To overcome this, we develop high-order analogs of the Poisson kernel, providing convergence rates for weak and pointwise approximation of generalized eigenfunctions for suitable classes of operators. We end the section with further examples.

\subsection{The Poisson kernel}\label{sec:poisson}

To illustrate, let $K_\epsilon(\lambda)$ be the Poisson kernel for the half-plane~\cite[p.~111]{stein2009real}. The functional calculus links~\cref{eqn:wave_packet} to the resolvent of $A$, as
\begin{equation}\label{eqn:stone}
\begin{split}
u_\lambda^{(\epsilon)} = \frac{1}{\pi}\int_\mathbb{R} \frac{\epsilon \dd E(\tilde\lambda)}{(\lambda-\tilde\lambda)^2+\epsilon^2}f &= \frac{1}{2\pi i}\int_\mathbb{R} \left[\frac{1}{\lambda-\tilde\lambda-i\epsilon}-\frac{1}{\lambda-\tilde\lambda+i\epsilon}\right]\dd E(\tilde\lambda)f\\
&= \frac{1}{2\pi i}\left[R_A(\lambda-i\epsilon)f-R_A(\lambda+i\epsilon)\right]f.
\end{split}
\end{equation}
For almost every point $\lambda\in\Lambda(A)$, the convolution on the left converges to a renormalized generalized eigenfunction, associated with $\lambda$, in the sense of distributions as $\epsilon\rightarrow 0$~\cite{derzko1972generalized}. On the other hand, the right-hand side can be computed directly by applying the resolvent to $f\in H$, i.e., by numerically solving the two linear operator equations $\smash{(A-(\lambda\pm i\epsilon))u_{\pm}=f}$ and taking the difference of the solutions. This relationship between the resolvent ``jump" and the generalized eigenfunctions is analogous to Stone's theorem for spectral measures~\cite{Stone}. To illustrate the basic features of this general computational approach, we examine several canonical examples: polynomial multiplication operators, constant coefficient differential operators, and perturbations thereof.

\subsubsection{Example 1: Multiplication operator}\label{sec:mult_poisson}

First, consider the multiplication operator $[Pu](x)=(x^3-x)\,u(x)$, where $x\in\Omega = (-1,1)$, from~\cref{sec:mult_ex}. Given $f\in C^\infty(\Omega)$, we use the change of variables $\hat\lambda=p(x)$ to explicitly compute the action of $u_\lambda^{(\epsilon)}$ on $\phi\in C^\infty(\Omega)$. We find that
\begin{equation}\label{eqn:mult_poisson_action}
\begin{split}
\langle u_\lambda^{(\epsilon)},\phi\rangle& = \frac{1}{2\pi i}\left\langle(P-\lambda+i\epsilon)^{-1}f-(P-\lambda-i\epsilon)^{-1}f,\phi\right\rangle\\
&=\frac{1}{\pi}\int_\Omega \frac{\epsilon f(x)\phi(x)}{(p(x)-\lambda)^2+\epsilon^2}\dd x 
= \frac{1}{\pi}\int_{\Lambda(P)} \sum_{k=1}^2\frac{\epsilon f(p_k^{-1}(\hat\lambda))\phi(p_k^{-1}(\hat\lambda))}{(\hat\lambda-\lambda)^2+\epsilon^2}\dd\hat\lambda.
\end{split}
\end{equation}
The right-hand side is simply the convolution of the density $\rho_{f,\phi}$ with the Poisson kernel and, therefore, converges pointwise to $\rho_{f,\phi}(\lambda)$ for all $\lambda\in\Omega\setminus\{0\}$. In other words, as $\epsilon\rightarrow 0$, we have that in the sense of distributions on $[-1,1]$,
\begin{equation}\label{eqn:mult_poisson_geig}
u_\lambda^{(\epsilon)}(x) \rightarrow \sum_{k=1}^2 f(p_k^{-1}(\lambda))\delta(x-p_k^{-1}(\lambda)), \qquad\text{for}\qquad \lambda\in\Omega\setminus\{0\}.
\end{equation}
We obtain a generalized eigenfunction of $P$ associated with $\lambda$, essentially a projection of $f$ onto this generalized eigenspace.

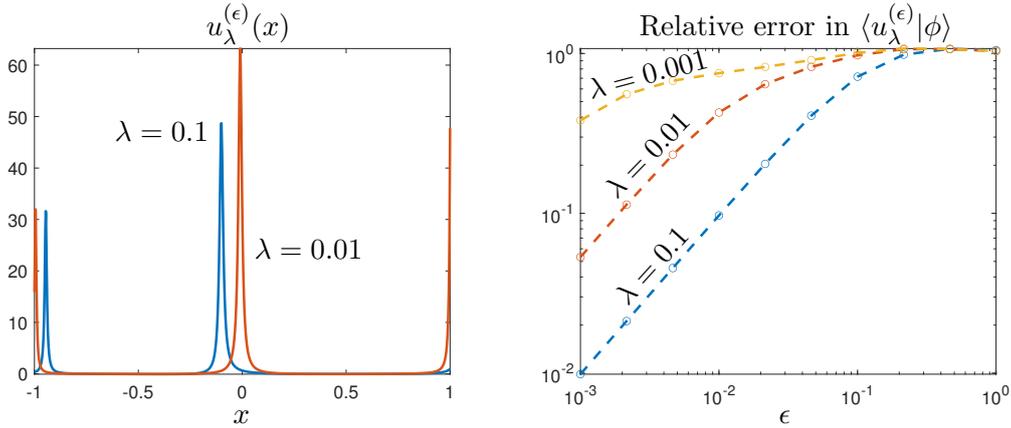
\begin{figure}[tbp!]
    \centering
    \begin{minipage}{0.48\textwidth}
        \begin{overpic}[width=\textwidth]{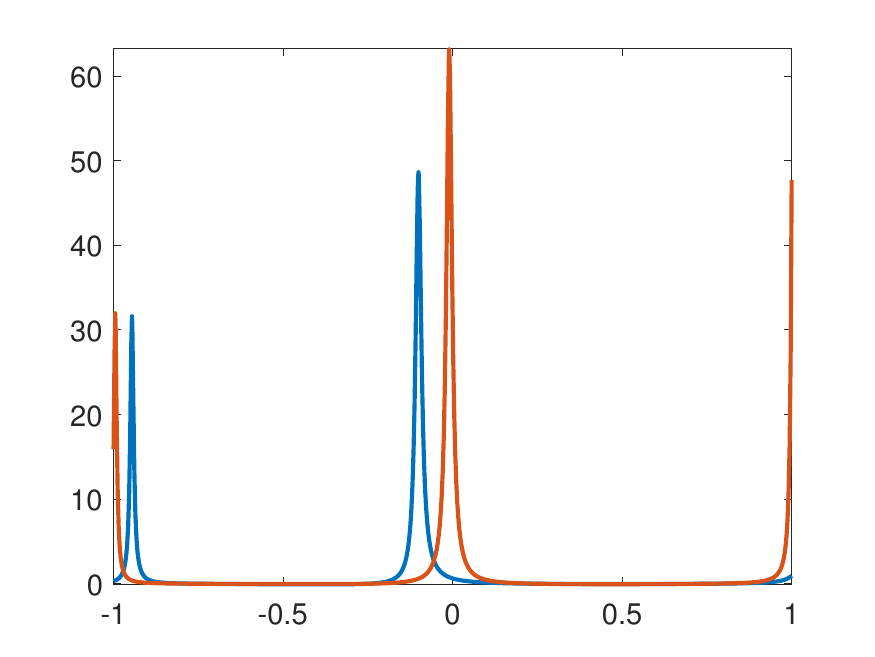}
        \put (50,-1) {$\displaystyle x$}
        \put (45,72) {$\displaystyle u_\lambda^{(\epsilon)}(x)$}
        \put (28,52) {{$\displaystyle \lambda = 0.1$}}
        \put (54,30) {{$\displaystyle \lambda = 0.01$}}
        \end{overpic}
    \end{minipage}
    \begin{minipage}{0.48\textwidth}
        \begin{overpic}[width=\textwidth]{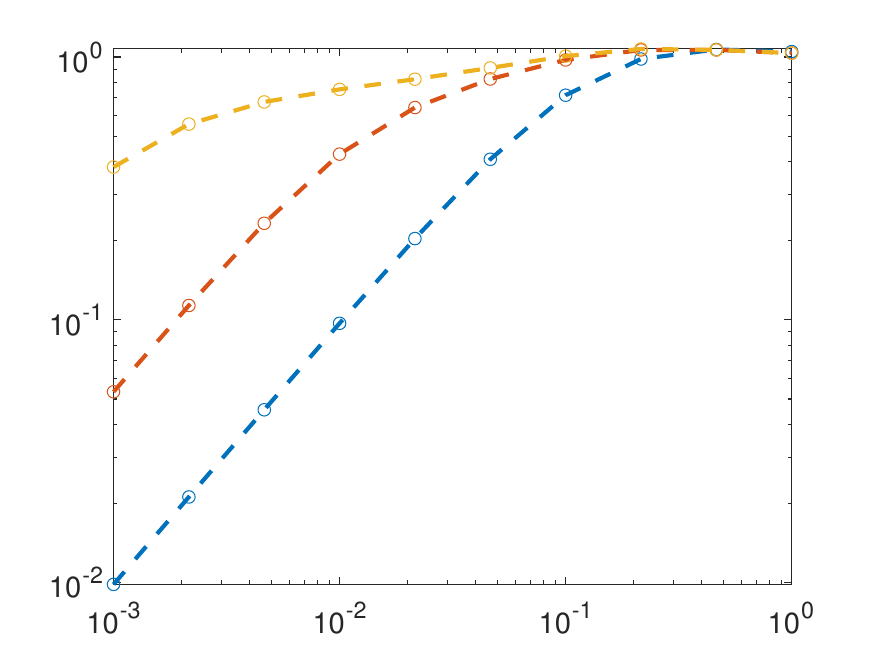}
        \put (50,-1) {$\displaystyle \epsilon$}
        \put (24,72) {$\displaystyle \text{Relative error in } \langle u_\lambda^{(\epsilon)}|\phi\rangle$}
      	\put (14,60) {\rotatebox{15} {$\displaystyle \lambda = 0.001$}}
        \put (17,41) {\rotatebox{46} {$\displaystyle \lambda = 0.01$}}
        \put (19,21) {\rotatebox{48} {$\displaystyle \lambda = 0.1$}}
        \end{overpic}
    \end{minipage}
    \caption{\label{fig:mult_op_poisson} Approximate generalized eigenfunctions $\smash{u_\lambda^{(\epsilon)}}$ with $\epsilon=0.01$ at $\lambda=0.1$ (blue) and $\lambda=0.01$ (red) are plotted in the left panel. The peaks correspond to locations of intersections of the plot in the left panel of \cref{fig:mult_op} with vertical lines at $\lambda$. Notice the ``ghost" of the generalized eigenfunction at $x=1$ appearing as $\lambda\rightarrow 0^+$, where the spectral measure $\rho_f$ has a jump discontinuity. The relative error in $\smash{\langle u_\lambda^{(\epsilon)}|\phi\rangle}$ is plotted against $\epsilon>0$ (right panel) for $\lambda=0.1$ (blue), $\lambda = 0.01$ (red), and $\lambda = 0.001$ (yellow) with $\phi=(1+x)\cos(\pi x)$.}
\end{figure}

Note that if we only assume that $f\in H= L^2(\Omega)$ (instead of $\Phi=C^\infty(\Omega)$), we can only conclude that convergence occurs for \textit{almost every} $\lambda\in (-1,1)$, i.e., at the Lebesgue points of the integrand in~\eqref{eqn:mult_poisson_action} (c.f.~\cref{thm:conv_ae}). However, convergence holds at every $\lambda\in\Omega\setminus\{0\}$ when $f\in C^\infty(\Omega)$ and can even be \textit{upgraded} as $\smash{u^{(\epsilon)}_\lambda}$ actually converges weakly in the sense of measures~\cite{billingsley1971weak}. Two wave packet approximations of generalized eigenfunctions of $P$, with $\lambda=0.1$ and $\lambda=0.01$, are shown in the left panel of~\cref{fig:mult_op_poisson}, computed using $\smash{f=(2+x)\cos(2\pi x)}$. The relative error in $\smash{\langle u_\lambda^{(\epsilon)}|\phi\rangle}$, with $\phi = (1+x)\cos(\pi x)$, is plotted against $10^{-3}\leq\epsilon\leq 1$ for $\lambda=0.1$, $0.01$, and $0.001$ in the right panel. The asymptotic convergence rate is $\mathcal{O}(\epsilon\log{\epsilon})$ and explicit bounds predict the larger observed errors when $\lambda$ is closer to the singularity of $\rho_f$ at $\lambda=0$ (see~\cref{sec:convergence}).

\subsubsection{Example 2: Differential operator}\label{sec:diff_poisson}

Next, we consider the differential operator $Au=-u''$ from~\cref{sec:diff_ex}. Computing $R_A(z)$ via Fourier transform, we find that
\begin{equation}\label{eqn:smoothed_diffopGEP}
u^{(\epsilon)}_\lambda=\frac{1}{2\pi i}\left[(A-\lambda+i\epsilon)^{-1}f-(A-\lambda-i\epsilon)^{-1}f\right] = \frac{1}{\pi}\int_{-\infty}^\infty\frac{\epsilon\hat f(k)}{(4\pi^2k^2-\lambda)^2+\epsilon^2}e^{2\pi i k x}\dd k.
\end{equation}
By taking an $L^2(\mathbb{R})$ inner product with $\phi\in S(\mathbb{R})$ and applying Fubini's theorem to interchange the order of integration, we see that
\begin{align*}
\langle u^{(\epsilon)}_\lambda | \phi\rangle &= \frac{1}{\pi}\int_{-\infty}^\infty\frac{\epsilon\hat f(k)\hat\phi(k)}{(4\pi^2 k^2-\lambda)^2+\epsilon^2}\dd k\\
&= \frac{1}{\pi}\int_0^\infty\frac{\epsilon\left[\hat f\left(\sqrt{y}/(2\pi)\right)\hat\phi\left(\sqrt{y}/(2\pi)\right) + \hat f\left(-\sqrt{y}/(2\pi)\right)\hat\phi\left(-\sqrt{y}/(2\pi)\right)\right]}{(y-\lambda)^2+\epsilon^2}\frac{\dd y}{4\pi\sqrt{y}}.
\end{align*}
Since $\hat f \hat\phi\in S(\mathbb{R})$ when $f,\phi\in S(\mathbb{R})$, the integrand on the right-hand side is smooth for each $\lambda>0$. Therefore, because the Poisson kernel is an approximation to the identity, the right-hand side converges pointwise as $\epsilon\rightarrow 0$ for each $\lambda>0$:
$$
\lim_{\epsilon\rightarrow 0}\,\langle u^{(\epsilon)}_\lambda|\phi\rangle = \frac{1}{4\pi\sqrt{\lambda}}\left[\hat f\left(\frac{\sqrt{\lambda}}{2\pi}\right)\hat\phi\left(\frac{\sqrt{\lambda}}{\pi}\right)+\hat f\left(\frac{-\sqrt{\lambda}}{2\pi}\right)\hat\phi\left(\frac{-\sqrt{\lambda}}{2\pi}\right)\right] = \rho_{f,\phi}(\lambda).
$$
Therefore, we again have convergence in the sense of distributions, i.e., 
\begin{equation}\label{eqn:wave_packet_approx}
u^{(\epsilon)}_\lambda(x) \rightarrow \frac{1}{4\pi\sqrt{\lambda}}\left[\hat f\left(\frac{\sqrt{\lambda}}{2\pi}\right)e^{i\sqrt{\lambda}x}+\hat f\left(\frac{-\sqrt{\lambda}}{2\pi}\right)e^{-i\sqrt{\lambda}x}\right]\quad\text{as}\quad \epsilon\rightarrow 0.
\end{equation}
As before, the wave packet converges to the projection of $f\in\Phi$ onto the generalized eigenspace associated with $\lambda$, which has multiplicity two in this case (if $\lambda\neq 0$). The factors of $\hat f(\pm\sqrt{\lambda}/(2\pi))/\sqrt{\lambda}$ make the computed generalized eigenfunctions orthogonal with respect to Lebesgue measure on the continuous spectrum rather than $\mu_f$, as in~\cref{sec:gefs}, and correspond to the appearance of $\rho_f(\lambda)$ in the convergence analysis in~\cref{sec:convergence}. Note that one could have also taken sines and cosines instead of complex exponentials. The generalized eigenfunction coordinates would then be given by sine and cosine transforms instead of the standard Fourier transform, but $u^{\epsilon}_\lambda$, the projection of $f$ onto the span of generalized eigenfunctions for $\lambda$, would not change.

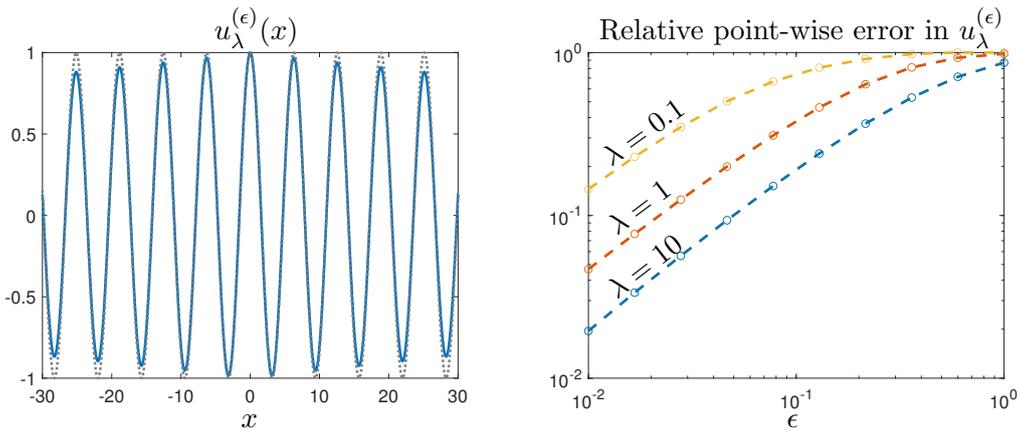
\begin{figure}[tbp!]
    \centering
    \begin{minipage}{0.48\textwidth}
        \begin{overpic}[width=\textwidth]{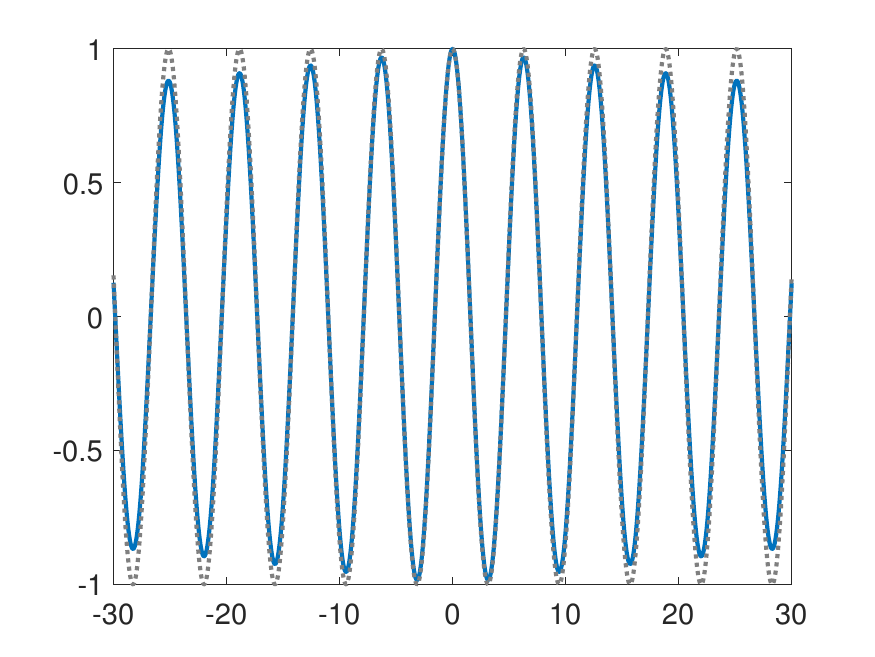}
        \put (50,-1) {$\displaystyle x$}
        \put (45,72) {$\displaystyle u_\lambda^{(\epsilon)}(x)$}
        \end{overpic}
    \end{minipage}
    \begin{minipage}{0.48\textwidth}
        \begin{overpic}[width=\textwidth]{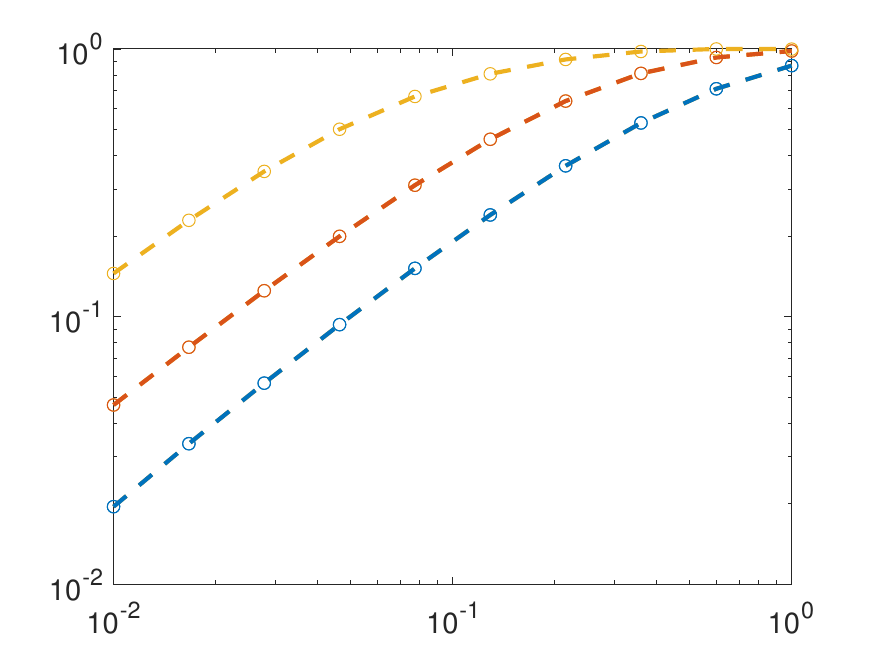}
        \put (50,-1) {$\displaystyle \epsilon$}
        \put (15,72) {$\displaystyle \text{Relative point-wise error in } u_\lambda^{(\epsilon)}$}
      	\put (15,48) {\rotatebox{36} {$\displaystyle \lambda = 0.1$}}
        \put (16,35) {\rotatebox{37} {$\displaystyle \lambda = 1$}}
        \put (16,23) {\rotatebox{38} {$\displaystyle \lambda = 10$}}
        \end{overpic}
    \end{minipage}
    \caption{\label{fig:diff_op_poisson} An approximate generalized eigenfunction $u_\lambda^{(\epsilon)}$ with $\epsilon=0.01$ at $\lambda=1$ (blue) is compared point-wise with the generalized eigenfunction $u_\lambda(x)=\cos(x)$ in the left panel (gray). The maximum pointwise relative error in $\smash{u_\lambda^{(\epsilon)}}$ over $x\in[-10,10]$ is plotted against $\epsilon$ (right panel) for $\lambda=10$ (blue), $\lambda = 1$ (red), and $\lambda = 0.1$ (yellow). The error increases as $\lambda$ approaches the endpoint of the spectrum $(\lambda=0)$ where $\rho_f$ is singular.}
\end{figure}

As in Example 1, note that if $\smash{f\in L^2(\mathbb{R})}$, convergence is only guaranteed for almost every $\lambda>0$, i.e., at the Lebesgue points of $\hat f(k)\hat \phi(k)$ (c.f.~\cref{thm:conv_ae})~\cite[p.~106]{stein2009real}. However, convergence can again be upgraded for $f\in S(\mathbb{R})$. Since $\hat f(k)e^{2\pi ikx}$ is itself a Schwartz function, convolution with the Poisson kernel after the change of variables converges pointwise in $x\in\mathbb{R}$, for each $\lambda>0$. \Cref{fig:diff_op_poisson} compares the generalized eigenfunction $\cos(\sqrt{\lambda}x)$ to the wave packet approximation computed with $f(x)=\exp(-x^2)/\sqrt{\pi}$ (left) and displays the maximum pointwise relative error of the approximation over the interval $[-10,10]$ as $\epsilon\rightarrow 0$ (right).

\subsection{Rational convolution kernels}\label{sec:rat_kern}

The wave packet approximations computed using the Poisson kernel converge slowly to generalized eigenfunctions as the smoothing parameter $\epsilon$ decreases (see~\cref{fig:mult_op_poisson,fig:diff_op_poisson}). For accurate wave packet approximations, one must take $\epsilon$ small and evaluate the resolvent $R_A(z)$ close to the spectrum of $A$. This leads to two numerical challenges. The associated linear systems typically require more computational degrees of freedom to solve accurately as $\epsilon\rightarrow 0^+$ and they also become increasingly ill-conditioned~\cite{colbrook2020computing,2023computingTI}.
To improve the convergence rates, we replace the Poisson kernel in~\cref{eqn:stone} with higher order kernels. By increasing the rate of convergence in $\epsilon$, we are able to compute accurate wave packet approximations without evaluating the resolvent close to the spectrum of $A$.

We consider a symmetric rational kernel in partial fraction form,
\begin{equation}\label{eqn:rat_kernel_form}
K(x)=\frac{1}{2\pi i}\sum_{j=1}^m\left[\frac{\alpha_j}{x-a_j}-\frac{\overline \alpha_j}{x-\overline a_j}\right],
\end{equation} 
where the poles $a_1,\ldots,a_m$ are distinct points in the upper half-plane and $\alpha_1,\ldots,\alpha_m$ are their associated complex residues. We say that $K$ is an $m$th order rational kernel if it satisfies the following normalization, zero moments, and decay conditions~\cite{colbrook2020computing}.

\begin{definition}[$m$th order kernel]
\label{def:mth_order_kernel}
Let $m$ be an integer $\geq 1$ and $K\in L^1(\mathbb{R})$. We say $K$ is an $m$th order kernel if it satisfies (i)-(iii):
\begin{itemize}
	\item[(i)] Normalized: $\int_{\mathbb{R}}K(x)dx=1$.
	\item[(ii)]  Zero moments: $K(x)x^j$ is integrable and $\int_{\mathbb{R}}K(x)x^jdx=0$ for $0<j<m$.
	\item[(iii)] Decay at $\pm\infty$: There is a constant $C_K$, independent of $x$, such that
	\begin{equation}
	\label{decay_bound}
	\left|K(x)\right|\leq \frac{C_K}{(1+\left|x\right|)^{m+1}}, \qquad x\in \mathbb{R}.
	\end{equation}
\end{itemize}
\end{definition}

The Poisson kernel in~\cref{eqn:stone} is a first-order rational kernel since it is integrable but does not have higher integrable or vanishing moments. For any distinct poles $a_1,\ldots,a_m$, we can compute residues $\alpha_1,\ldots,\alpha_m$ so that the kernel in~\eqref{eqn:rat_kernel_form} is an $m$th order kernel. This is done by solving the system~\cite{colbrook2020computing}
\begin{equation}\label{eqn:vandermonde_condition}
\begin{pmatrix}
1 & \dots & 1 \\
a_1 & \dots & a_m \\
\vdots & \ddots & \vdots \\
a_1^{m-1} &  \dots & a_m^{m-1}
\end{pmatrix}
\begin{pmatrix}
\alpha_1 \\ \alpha_2\\ \vdots \\ \alpha_m
\end{pmatrix}
=\begin{pmatrix}
1 \\ 0 \\ \vdots \\0
\end{pmatrix}.
\end{equation}
For $m\geq 2$, $m$th order rational kernels can provide higher orders of accuracy than the Poisson kernel as $\epsilon\rightarrow 0$ in~\cref{eqn:stone}. The functional calculus links~\cref{eqn:wave_packet} to the resolvent again, but in place of~\cref{eqn:stone}, we obtain
\begin{equation}\label{eqn:mth_stone}
u^{(\epsilon)}_\lambda = \sum_{k=1}^m \alpha_k R_A(\lambda+\epsilon a_k)v - \overline{\alpha}_k R_A(\lambda+\epsilon \overline{a}_k)v. 
\end{equation}
The approximate generalized eigenfunction, $\smash{u^{(\epsilon)}_\lambda}$, is again computed by numerically solving a sequence of (discretized) operator equations and taking linear combinations.

\begin{figure}
    \centering
    \begin{minipage}{0.48\textwidth}
        \begin{overpic}[width=\textwidth]{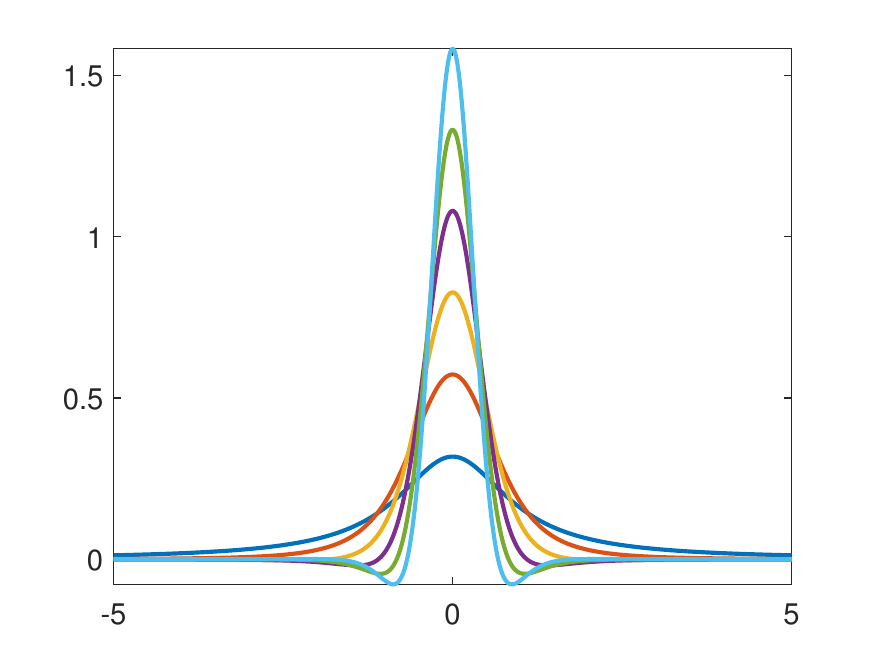}
        \put (50,-2) {$\displaystyle x$}
        \put (20,72) {$\displaystyle m\text{th order rational kernels}$}
        \put (62,20) {{$\displaystyle m = 1\text{ (Poisson)}$}}
        \put (62,28) {{$\displaystyle m = 2$}}
        \put (62,37) {{$\displaystyle m = 3$}}
        \put (62,46) {{$\displaystyle m = 4$}}
        \put (62,55) {{$\displaystyle m = 5$}}
        \put (62,64) {{$\displaystyle m = 6$}}
        \end{overpic}
    \end{minipage}
    \begin{minipage}{0.48\textwidth}
        \begin{overpic}[width=\textwidth]{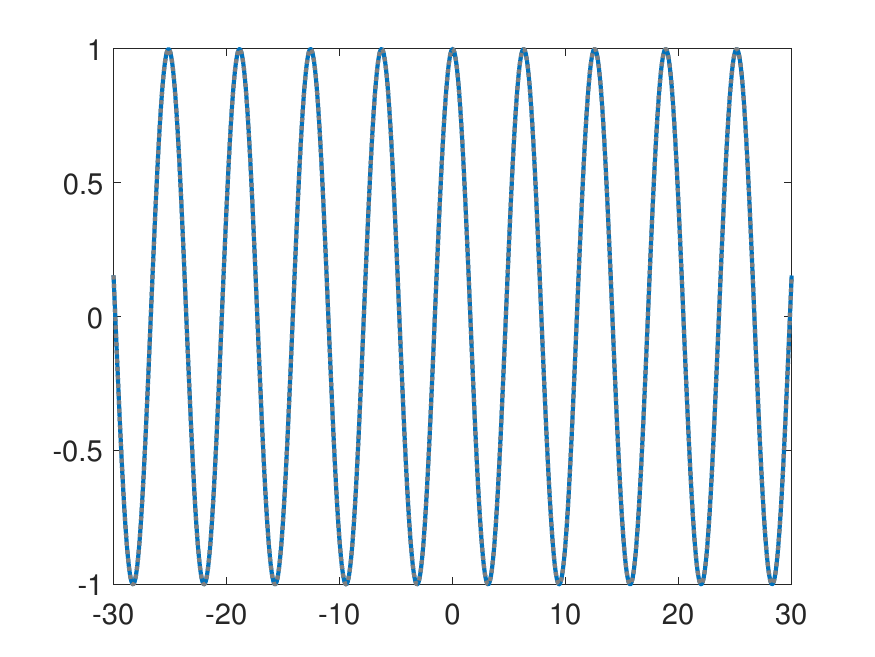}
        \put (50,-1) {$\displaystyle x$}
        \put (45,72) {$\displaystyle u_\lambda^{(\epsilon)}(x)$}
        \end{overpic}
    \end{minipage}
    \caption{\label{fig:mth_order_kernel} Rational kernels of order $m=1,\ldots,6$ with equispaced poles in $[-1+i,1+i]$ are plotted in the left panel. A higher order approximation ($m=3$) of the generalized eigenfunction in~\cref{fig:diff_op_poisson} improves the point-wise accuracy over a large interval with the same value of $\epsilon=0.01$.}
\end{figure}

\begin{figure}[tbp!]
    \centering
    \begin{minipage}{0.48\textwidth}
        \begin{overpic}[width=\textwidth]{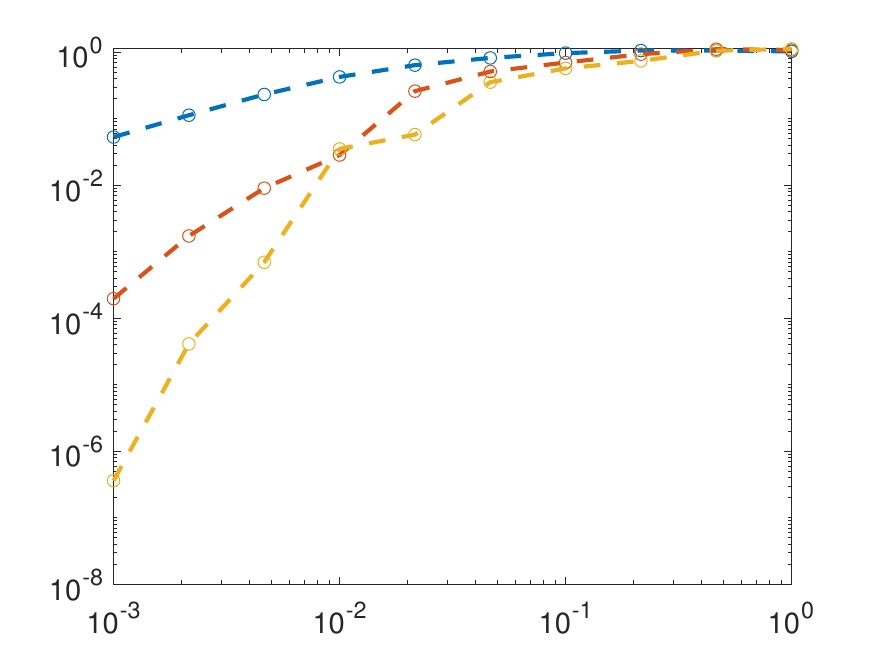}
        \put (50,-1) {$\displaystyle \epsilon$}
        \put (23,72) {$\displaystyle \text{Relative error in } \langle u_\lambda^{(\epsilon)}|\phi\rangle$}
      	\put (16,55) {\rotatebox{18} {$\displaystyle m = 1$}}
        \put (16,39) {\rotatebox{39} {$\displaystyle m = 3$}}
        \put (16,23) {\rotatebox{53} {$\displaystyle m = 5$}}
        \end{overpic}
    \end{minipage}
    \begin{minipage}{0.48\textwidth}
        \begin{overpic}[width=\textwidth]{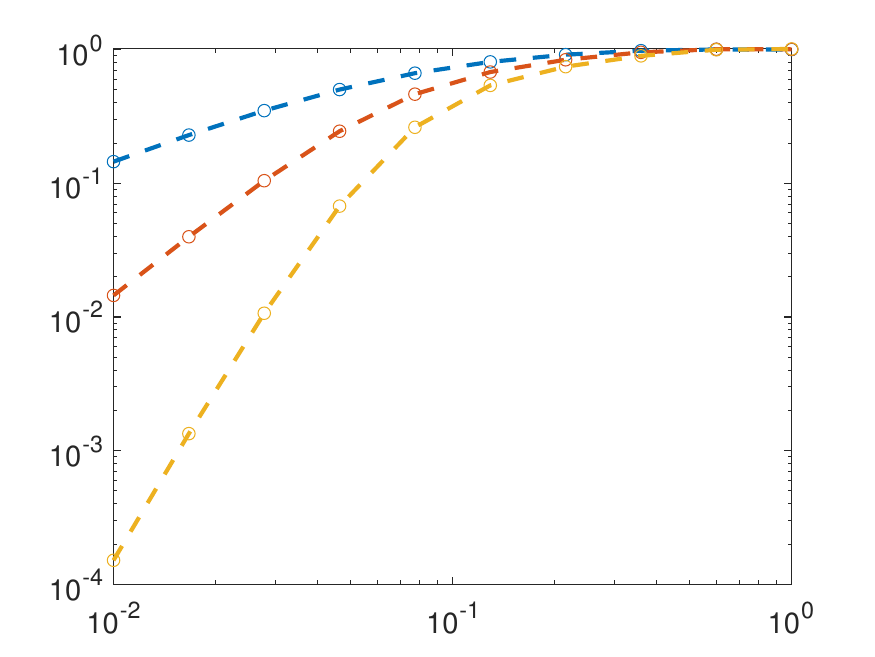}
        \put (50,-1) {$\displaystyle \epsilon$}
        \put (15,72) {$\displaystyle \text{Relative point-wise error in } u_\lambda^{(\epsilon)}$}
      	\put (17,60) {\rotatebox{18} {$\displaystyle m = 1$}}
        \put (16,47) {\rotatebox{37} {$\displaystyle m = 2$}}
        \put (16,24) {\rotatebox{55} {$\displaystyle m = 3$}}
        \end{overpic}
    \end{minipage}
    \caption{\label{fig:mth_order_accuracy} The relative error in $\smash{\langle u_\lambda^{(\epsilon)}|\phi\rangle}$ at $\lambda=0.01$, for the multiplication operator in~\cref{sec:mult_poisson} with $\phi=(1+x)\cos(\pi x)$, is plotted against $\epsilon$ (left panel) for kernel orders $m = 1$ (blue), $m = 3$ (red), and $m = 5$ (yellow). The maximum point-wise relative error in $\smash{\mu_\lambda^{(\epsilon)}}(x)$ for $x\in[-10,10]$, at $\lambda=0.1$, is plotted against $\epsilon$ (right panel) for the differentiation operator in~\cref{sec:diff_poisson} for $m=1$ (blue), $m=2$ (red), $m=3$ (yellow).}
\end{figure}
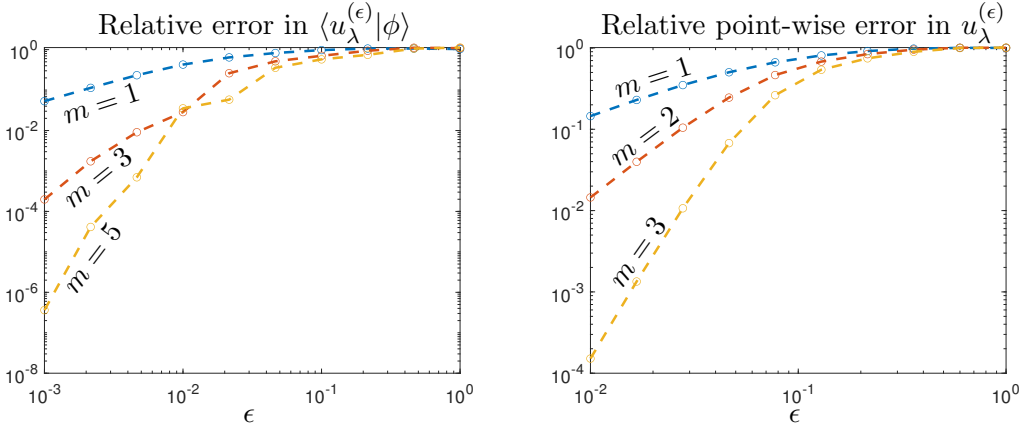

The first six rational kernels (of orders $m=1,\ldots,6$) with equispaced poles in $[-1+i,1+i]$ are plotted in the left panel of~\cref{fig:mth_order_kernel}. A wave packet approximation of the generalized eigenfunction in~\cref{fig:diff_op_poisson}, computed with a $3^{\rm rd}$-order rational kernel, is plotted in the right panel of~\cref{fig:mth_order_kernel}. The higher order approximation is point-wise accurate over a significantly larger interval than the Poisson wave packet approximation with the same value of $\epsilon=0.01$. Figure~\ref{fig:mth_order_accuracy} compares weak (left panel) and point-wise (right panel) convergence rates of higher order wave packet approximations for the multiplication and differential operators from~\cref{sec:mult_poisson} and~\cref{sec:diff_poisson}, respectively. Note the improved convergence rates, which allow us to compute accurate wave packet approximations without solving computationally expensive linear systems at smaller $\epsilon$.

\subsection{Convergence theory}\label{sec:convergence}

We now study the convergence and approximation error for the approximate generalized eigenfunction $\smash{u^{(\epsilon)}_\lambda}$ defined in~\cref{eqn:mth_stone}. First, we show that for any $\phi\in\Phi^*$, the generalized eigenfunction coordinates of $\phi$ converge at almost every point in the continuous spectrum, i.e., that $\smash{\langle u^{(\epsilon)}_\lambda,\phi\rangle\rightarrow\rho_f(\lambda)\langle u_\lambda | \phi\rangle}$ as $\epsilon\rightarrow 0$. The appearance of $\rho_f(\lambda)$ indicates that the wave packet approximations converge to a re-normalized set of generalized eigenfunctions on the absolutely continuous spectrum. The measure $\mu_f$ in the Parseval relation in~\cref{eqn:completeness} is replaced by Lebesgue measure and the density function $\rho_f(\lambda)$ is absorbed into the wave packet approximations to the generalized eigenfunctions. The proofs are in~\cref{sec:appendix}.

\begin{theorem}\label{thm:conv_ae}
Let $A:D(A)\rightarrow H$ be a selfadjoint operator on a rigged Hilbert space $\Phi\subset H\subset\Phi^*$, where $\Phi$ is a countably Hilbert nuclear space and $A\Phi\subset\Phi$. Let $K$ be an $m$th order rational kernel satisfying $(i)-(iii)$ in~\cref{def:mth_order_kernel} and~\cref{eqn:rat_kernel_form,eqn:vandermonde_condition}, and, given $f\in\Phi$, let $\smash{u_\lambda^{(\epsilon)}}$ be the associated wave packet approximation in~\cref{eqn:mth_stone}. If the spectrum of $A$ is absolutely continuous in the interval $\Omega=[a,b]$, then for each $\phi\in\Phi$,
$$
\lim_{\epsilon\rightarrow 0^+}\langle u_\lambda^{(\epsilon)},\phi\rangle = \rho_f(\lambda)\langle u_\lambda | \phi\rangle\qquad\text{for a.e.-}\lambda\in{\rm int}(\Omega),
$$
where, $u_\lambda\in\Phi^*$ is a generalized eigenfunction of $A$ satisfying~\cref{eqn:def_geneig}.
\end{theorem}
\begin{proof}
See~\cref{sec:appendix}.
\end{proof}

In particular, the convergence in~\cref{thm:conv_ae} holds at every point of continuity of $\rho_{f,\phi}$ and, more generally, on the Lebesgue set of $\rho_{f,\phi}$ (see~\cref{sec:appendix}). For practical numerical computation, understanding the approximation error and convergence rate for specific test functions in $\Phi$ is crucial. The rational convolution kernels achieve higher orders of convergence in regions of the absolutely continuous spectrum where $\rho_{f,\phi}$ is locally smooth. We use H\"older continuity in this context to establish the convergence rates. The space of functions with $n\geq 0$ continuous derivatives on a closed interval $\Omega = [a,b]$ ($a<b$) and with an $\alpha$-H\"older continuous $n$th derivative is denoted by $C^{n,\alpha}(\Omega)$.

\begin{theorem}[Weak-$^*$ convergence rates]\label{thm:conv_rates}
Suppose that the hypotheses of~\cref{thm:conv_ae} hold and, additionally, suppose that the restriction of the Radon-Nikodym derivative $\rho_{f,\phi}$ to the interval $\Omega=[a,b]$ is in $C^{n,\alpha}(\Omega)$. Then for any fixed $\lambda\in{\rm int}(\Omega)$,
$$
\left|\langle u_\lambda^{(\epsilon)},\phi\rangle - \rho_f(\lambda)\langle u_\lambda | \phi\rangle\right| = \mathcal{O}(\epsilon^{n+\alpha}) + \mathcal{O}(\epsilon^m\log(\epsilon)),\qquad\text{as}\qquad\epsilon\rightarrow 0.
$$
\end{theorem}
\begin{proof}
See~\cref{sec:appendix}.
\end{proof}

The improved convergence rate for $m\geq 2$ as $\epsilon\rightarrow 0$ allows us to accurately approximate generalized eigenfunctions without taking $\epsilon>0$ very small, i.e., without evaluating $R_A(z)$ at points very close to the spectrum of $A$. This improves the algorithmic efficiency and numerical stability of the wave packet approximation scheme.

As examples~\cref{sec:mult_poisson,sec:diff_poisson} illustrate, the weak limit in~\cref{eqn:def_geneig} sometimes holds in a finer topology, and the generalized eigenfunctions may even be identified with elements of a classical function space. This phenomenon is typical of examples involving elliptic partial differential operators because locally smooth generalized eigenfunctions (and smooth wave packet approximations) are implied by regularity theory. Our next result provides, for locally continuous and locally bounded generalized eigenfunctions, sufficient conditions for the wave packet approximations to converge locally uniformly. 

Suppose that $A$ acts on a space of square-integrable functions $L^2(S,\nu)$, where $S$ is a metric space with strictly positive measure $\nu$ (recall that $\nu$ is strictly positive if every open set is $\nu$-measurable and every non-empty open set has nonzero $\nu$-measure). We call $K\subset S$ a compact domain if it is the closure of a nonempty open connected set with bounded diameter. If $\Phi|_K$ denotes the set of $\phi\in\Phi$ with compact support in $K$ and $\Phi|_K$ is dense in $L^2(K,\nu)\subset L^2(S,\nu)$, then $\Phi|_K$ separates points in $C(K)$: if $f\in C(K)$ and $\int_K \phi\,f\,\dd\nu=0$ for every $\phi\in\Phi|_K$, then $f=0$. In this case, we say that $\psi\in\Phi^*$ is continuous on $K$ if there is a continuous function $\psi|_k\in C(K)$ such that $\langle \psi | \phi \rangle = \int_K \phi\,\psi|_k\,\dd\nu$ for all $\phi\in \Phi|_K$. Given a closed interval $\Omega\in\mathbb{R}$, we say that $u:K\times\Omega \rightarrow \mathbb{C}$ is $K$-uniformly continuous on $\Omega$ if for any $\delta>0$ there is a $\delta'>0$ such that $|\lambda-\lambda'|<\delta'$ implies $|u(x,\lambda)-u(x,\lambda')|<\delta$ for all $x\in K$. Similarly, $u$ is $K$-uniformly $(n,\alpha)$-H\"older continuous on $\Omega$ if $\smash{\|u(x,\cdot)\|_{C^{n,\alpha}(\Omega)}}$ is uniformly bounded on $K$.

\begin{theorem}[Uniform convergence on compact sets]\label{thm:conv_point}
Let the hypotheses of~\cref{thm:conv_ae} hold with $H=L^2(S,\nu)$ for metric space $S$ and strictly positive $\nu$. Given compact domain $K\subset S$, suppose $\Phi|_K$ is dense in $L^2(K,\nu)$ and $u_\lambda$ is continuous on $K$ for all $\lambda\in\mathbb{R}$ with $\sup_{\lambda\in\mathbb{R}}\|u_\lambda|_K\|_{C(K)}<\infty$. Given $\lambda_*\in{\rm int}(\Omega)$ and $\delta>0$ such that $I=[\lambda_*-\delta,\lambda_*+\delta]\in{\rm int}(\Omega)$, if $(x,\lambda) \rightarrow \rho_f(\lambda) u_\lambda|_k(x)$ is $K$-uniformly continuous on $I$ and $\smash{u_{\lambda_*}^{(\epsilon)}|_K\in C(K)}$,
$$
\lim_{\epsilon\rightarrow 0^+} \sup_{x\in K}\lvert u_{\lambda_*}^{(\epsilon)}(x) - \rho_f(\lambda_*)u_{\lambda_*}(x)\rvert = 0.
$$
Moreover, if $(x,\lambda)\rightarrow \rho_f(\lambda)u_\lambda|_K(x)$ is $K$-uniformly $(n,\alpha)$-H\"older continuous on $I$ with nonnegative integer $n$ and $\alpha\in[0,1]$, then
$$
\sup_{x\in K}\lvert u_{\lambda_*}^{(\epsilon)}(x) - \rho_f(\lambda_*)u_{\lambda_*}(x)\rvert = \mathcal{O}(\epsilon^{n+\alpha}) + \mathcal{O}(\epsilon^m\log(\epsilon)),\qquad\text{as}\qquad\epsilon\rightarrow 0^+.
$$
\end{theorem}
\begin{proof}
See~\cref{sec:appendix}.
\end{proof}

Note that~\cref{thm:conv_point} may be applied to both $L^2(\mathbb{R}^d)$ and $\ell^2$, the space of square-summable sequences, as well as weighted variants of these spaces. For example, in the case of $\ell^2$, $S=\mathbb{N}$ is equipped with the open sets of the discrete topology so that the counting measure is strictly positive. For both $L^2(\mathbb{R}^d)$ and $\ell^2$, the standard rigging employs nuclear spaces ($\mathcal{S}$ and the space of rapidly decaying sequences, respectively) that satisfy the density requirement in $L^2(K,\nu)$ for compact domains $K$. Consequently,~\cref{thm:conv_point} provides useful criteria for convergent wave packet approximations to scattering states of discrete Hamiltonian operators on lattices in condensed matter physics~\cite{2023computingTI}.

Even when the generalized eigenfunctions are not elements of classical function spaces (at least, locally), the wave packet approximations may still converge in a stronger topology than that of $\Phi^*$. For example, the wave packet approximations to the multiplication operator's Dirac delta eigenfunctions in~\cref{sec:mult_poisson} converge weakly in the sense of measures. Below,~\cref{thm:conv_abs} provides a general framework for understanding how and when the wave packet's convergence may be ``upgraded" to a stronger topology. 

A topological space is called a Suslin space if it is the image of a Polish space (a separable, completely metrizable topological space) under a continuous bijection~\cite{thomas1975integration}. As Thomas notes~\cite{thomas1975integration}, many common separable, locally convex topological vector spaces (TVS) used in applied analysis are Suslin spaces. For example, separable Frechet, Banach, and Hilbert spaces are Suslin, as are their duals, and the usual spaces of test functions (e.g. Schwartz, smooth compactly supported, and holomorphic spaces of test functions) and the associated spaces of distributions are Suslin. We state~\cref{thm:conv_abs} in generality and then describe its application to the concrete examples of this paper.

If $F$ is a locally convex Suslin TVS and $\Omega\subset\mathbb{R}$, we say that a vector-valued map $u:\Omega\rightarrow F$ has bounded image in $F$ if any neighborhood of the origin can be scaled to include the image~\cite[p.~8]{rudin}. Equivalently, $\sup_{\lambda\in\Omega}|\langle \psi | u(\lambda)\rangle|<\infty$ for every linear functional $\psi\in F^*$\cite[Thm~3.18]{rudin}. If $F$ is a Frechet space, this is equivalent to $\sup_{\lambda\in\Omega}|u(\lambda)|_p<\infty$ for any countable family of seminorms $\{|\cdot|_p\}_{p=1}^\infty$ that separates points in $F$~\cite[Thm~1.37]{rudin}. Consequently, in a Banach space one need only check that the image is norm bounded.

\begin{theorem}[Convergence in Suslin spaces]\label{thm:conv_abs}
Let the hypotheses of~\cref{thm:conv_ae} hold and, given a locally convex Suslin TVS, $F$, suppose that $\lambda\rightarrow u_\lambda$ has bounded image in $F$. Additionally, suppose there is a continuous embedding $i:\Phi\rightarrow F^*$ such that $\langle i(\phi) | f \rangle = \langle f | \phi\rangle$ when $f\in F \bigcap \Phi^*$, $\phi\in\Phi$, and that $i(\Phi)\subset F^*$ separates points in $F$. Given any $\lambda_*\in{\rm int}(\Omega)$ and $\psi\in F^*$, if $\smash{u_{\lambda_*}^{(\epsilon)}\in F}$ and $\rho_f(\lambda)\langle\psi | u_\lambda\rangle$ is continuous at $\lambda_*$, then
$$
\lim_{\epsilon\rightarrow 0^+}\langle \psi | u_{\lambda_*}^{(\epsilon)}\rangle = \rho_f(\lambda_*)\langle \psi | u_{\lambda_*}\rangle.
$$
Moreover, given $\delta>0$ such that $I=[\lambda_*-\delta,\lambda_*+\delta]\in{\rm int}(\Omega)$, if $\lambda\rightarrow \rho_f(\lambda)\langle \psi | u_{\lambda}\rangle$ is $(n,\alpha)$-H\"older continuous on $I$ with nonnegative integer $n$ and $\alpha\in[0,1]$, then
$$
\lvert\langle \psi | u_{\lambda_*}^{(\epsilon)}\rangle - \rho_f(\lambda_*)\langle \psi | u_{\lambda_*}\rangle\rvert = \mathcal{O}(\epsilon^{n+\alpha}) + \mathcal{O}(\epsilon^m\log(\epsilon)),\qquad\text{as}\qquad\epsilon\rightarrow 0^+.
$$
Finally, the conclusions hold if $\langle \psi|\cdot\rangle$ is replaced by any continuous, $\mu_f$-integrable seminorm $|\cdot|$ on $F$.
\end{theorem}
\begin{proof}
See~\cref{sec:appendix}.
\end{proof}

To illustrate the application of~\cref{thm:conv_abs}, consider the convergence of wave packet approximations in~\cref{sec:diff_poisson}. First, the generalized eigenfunctions of the second derivative operator are elements of the Frechet space $C_b^\infty(\mathbb{R})$ equipped with the countable family of seminorms $|f|_{j,n} = \sup_{|x|\leq n}|f^{(j)}(x)|$. The generalized eigenfunctions, sines and cosines normalized so that they are $\mu_f$-integrable, are uniformly bounded in each of these seminorms. The Schwartz space $\mathcal{S}(\mathbb{R})\subset L^2(\mathbb{R})$ defines a separating set of continuous linear functionals on $C_b^\infty(\mathbb{R})$ via $\langle \phi | f\rangle = \int_\mathbb{R} \phi\, f\,\mathrm{d}x$ and this is compatible with the dual pairing between $\Phi^*$ and $\Phi$ because $f$ induces a continuous linear functional on $\Phi$ in an analogous manner. By Fourier analysis, $\rho_f(\lambda)$ is smooth away from $\lambda=0$, so~\cref{thm:conv_abs} guarantees convergence in each seminorm: $\smash{|u_\lambda^{(\epsilon)}|_{j,n}\rightarrow \rho_f(\lambda)|u_\lambda|_{j,k}}$. The conclusion of~\cref{thm:conv_point} is strengthened; the continuity of $u_{\lambda_*}$ and its derivatives is leveraged and each derivative converges uniformly on the compact intervals $K_n=[-n,n]$. 

Alternatively, consider the convergence of wave packet approximations for the multiplication operator in~\cref{sec:mult_poisson}. The generalized eigenfunctions, Dirac delta distributions, are elements of the space of finite, positive measures on $[-1,1]$, denoted $M([-1,1])$, equipped with the topology of weak convergence of measures. This topological vector space is locally convex and Suslin because it is separable and completely metrizable via the L\'evy--Prokhorov metric~\cite[p.~72]{billingsley1999convergence}. Each element of the nuclear space $\phi\in C^\infty([-1,1])$ naturally induces a continuous linear functional on $M([-1,1])$ via $\int_{-1}^1 \phi(x)\,\mathrm{d}\mu(x)$, for all $\mu\in M([-1,1])$ and the collection of all such functionals separates points. Therefore, in regions of the spectrum of the multiplication operator where the density $\rho_f(\lambda)$ is continuous, the wave packet approximations converge weakly in the sense of measures.

\subsection{Examples}\label{sec:more_example}

This section explores the generalized eigenfunctions and spectral measures of differential and integral operators using the high-order wave packet approximations introduced in~\cref{sec:rat_kern} and studied in~\cref{sec:convergence}.

\subsubsection{Example 3: Integral operator}\label{sec:integral_example}

Consider a trace-class integral perturbation of the cubic multiplication operator in~\cref{sec:mult_ex}, given by
$$
[Au](x) = (x^3-x)u(x) + \int_{-1}^1 e^{-(x^2+y^2)}u(y)\dd y.
$$
The absolutely continuous spectrum of $A$ fills the unit interval, as in the unperturbed case. However, the perturbation alters the generalized eigenfunctions and spectral measures of $A$. In fact, since the perturbation is a rank-one operator on $L^2(\Omega)$, where $\Omega = (-1,1)$, the resolvent of $A$ can be computed explicitly with a continuous analogue of the Sherman-Morrison matrix identity~\cite{sherman1950adjustment,hager1989updating}. Denoting the unperturbed operator by $\tilde A$ and  $g(x)=\exp(-x^2)$ we calculate that, for $z\not\in\Lambda(A)$,
$$
(A-z)^{-1}f =  \frac{1}{x^3-x-z}\left[f + \frac{\langle R_{\tilde A}(z)g|f\rangle}{1+\langle R_{\tilde A}(z)g|g\rangle}g\right]\qquad\forall f\in L^2(\Omega).
$$
Considering Stone's formula, it is clear that while the generalized eigenfunctions of $A$ are no longer simple combinations of Dirac delta distributions, they remain dominated by singularities at the roots of the cubic family, $x^3-x-\lambda$, as before. 

\begin{figure}[tbp!]
    \centering
    \begin{minipage}{0.48\textwidth}
        \begin{overpic}[width=\textwidth]{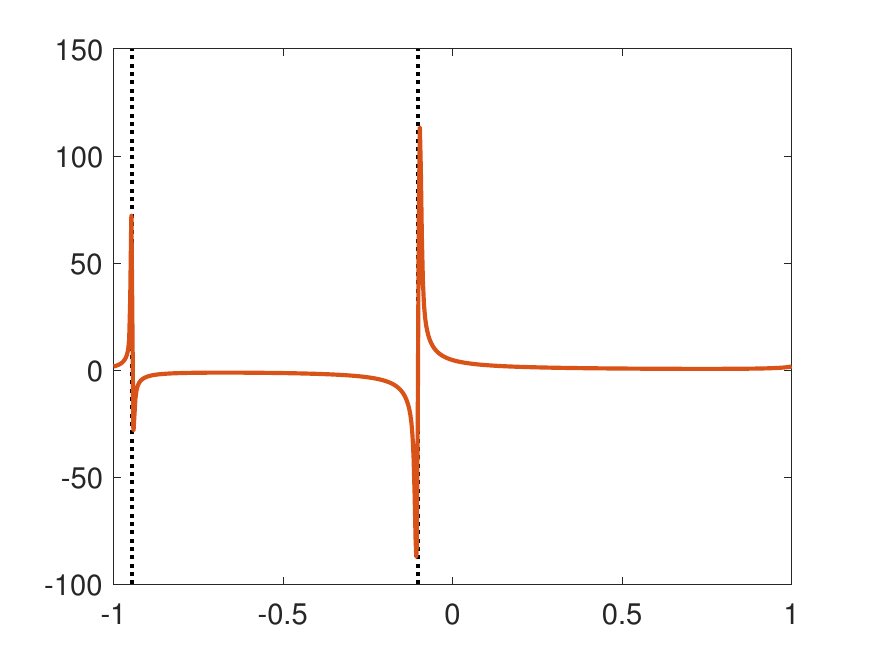}
        \put (50,-1) {$\displaystyle x$}
        \put (47,72) {$\displaystyle \psi_\lambda^{(\epsilon)}(x)$}
        \end{overpic}
    \end{minipage}
    \begin{minipage}{0.48\textwidth}
        \begin{overpic}[width=\textwidth]{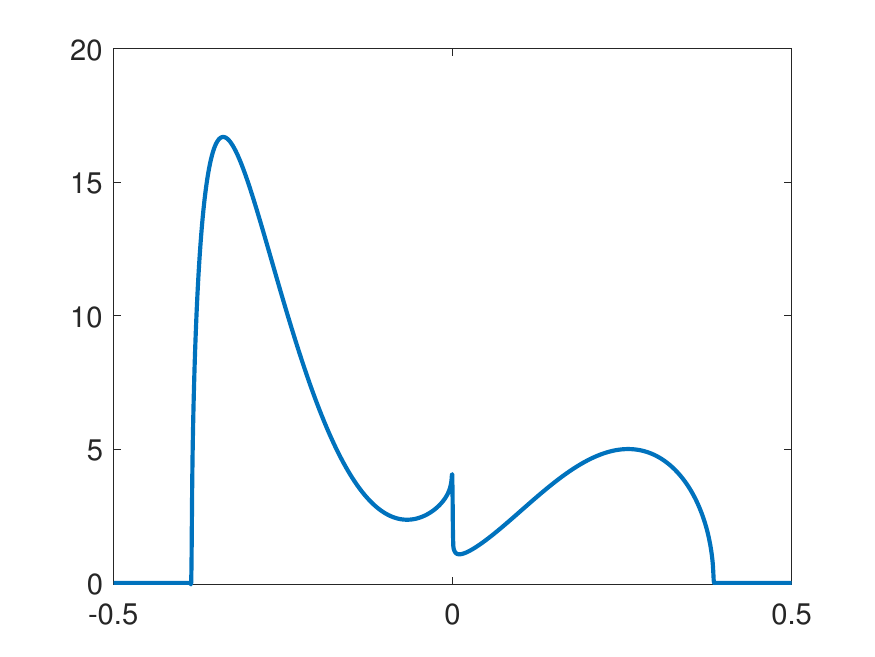}
        \put (50,-1) {$\displaystyle \epsilon$}
        \put (47,72) {$\displaystyle \mu_f^{(\epsilon)}(\lambda)$}
        \end{overpic}
    \end{minipage}
    \caption{\label{fig:perturbed_cubic} A $4^{\rm th}$-order wave packet approximation (left panel) to a generalized eigenfunction of the perturbed multiplication operator in~\cref{sec:integral_example} with $\lambda=0.1$ and $\epsilon=0.01$. The dashed vertical lines mark the location of the unperturbed operator's Dirac delta eigenfunctions, i.e., the roots of the cubic $x^3-x-\lambda$. The right panel displays a $4^{\rm th}$ order approximation to the spectral measure of the perturbed operator.}
\end{figure}

The left panel of \Cref{fig:perturbed_cubic} shows a $4^{\rm th}$-order wave packet approximation to the $\lambda=0.1$ generalized eigenfunction of $A$ with smoothing parameter set to $\epsilon=0.01$. The wave packet approximation has sharp peaks at the roots of the cubic $x^3-x-\lambda$ and these become taller and narrower as $\epsilon\rightarrow 0$. A $4^{\rm th}$ order smoothed approximation to the spectral measure of $A$, with respect to $\phi(x)=(2+x)\cos(2\pi x)$, is shown in the the right-hand panel. The spectral measure appears to have algebraic singularities at the endpoints of the spectrum but, in contrast to the unperturbed case, is bounded there (c.f., right panel of ~\cref{fig:mult_op}). 

\subsubsection{Example 4: 1D Schr\"odinger operators}

The generalized eigenfunctions of Schr\"odinger operators, often called scattering states, play an important role in the theory of quantum mechanical scattering. A one-dimensional Schr\"odinger operator with a smooth, integrable potential has the form
\begin{equation}\label{eqn:schrodinger_1D}
[Au](x) = -\frac{\mathrm{d}^2u}{\mathrm{d}x^2} + v(x)u, \qquad\text{where}\qquad u\in H^2(\mathbb{R}).
\end{equation}
Here, the potential satisfies $v\in L^1(\mathbb{R})\cap C^\infty(\mathbb{R})$ and $H^2(\mathbb{R})$ indicates the Sobolev space of functions with two square-integrable weak derivatives. The spectrum of $A$ is absolutely continuous on the positive real axis with multiplicity $m=2$ and the asymptotics of the associated generalized eigenfunctions as $x\rightarrow\pm\infty$ are oscillatory~\cite{agmon1975spectral}.

\Cref{fig:schrodinger_short,fig:schrodinger_long} show $4^{\rm th}$-order wave packet approximations to generalized eigenfunctions of~\cref{eqn:schrodinger_1D} equipped with a short-range and a long-range potential, respectively. The potential functions are plotted in the left panel of the corresponding figure. Our numerical scheme (orange curves, right panels) achieves high-order accuracy in both cases by combining the high-order accurate wave packet approximations and banded, spectrally accurate discretizations of the resolvent of~\cref{eqn:schrodinger_1D} on the whole real line \cite{christov1982complete,iserles2020family}. We compare with approximations computed using a Lippmann-Schwinger formulation of the generalized eigenvalue problem followed by careful domain truncation with a perfectly matched layer (black curves). Approximations based on domain truncation and asymptotics may converge slowly in the truncation parameter for long-range potentials like the one in~\cref{fig:schrodinger_long} (left panel).

\begin{figure}[tbp!]
    \centering
    \begin{minipage}{0.48\textwidth}
        \begin{overpic}[width=\textwidth]{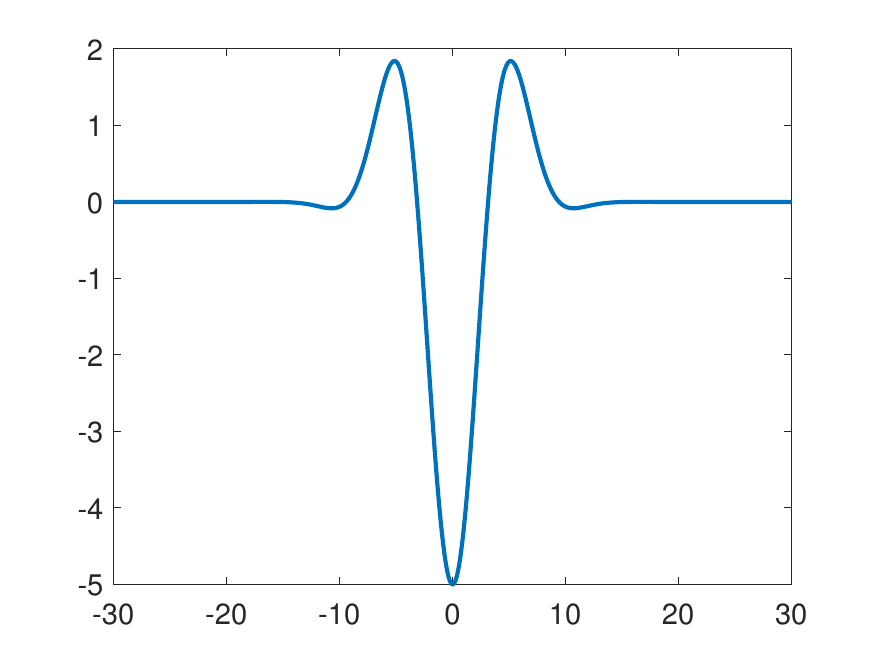}
        \put (50,-1) {$\displaystyle x$}
        \put (47,72) {$\displaystyle v(x)$}
        \end{overpic}
    \end{minipage}
    \begin{minipage}{0.48\textwidth}
        \begin{overpic}[width=\textwidth]{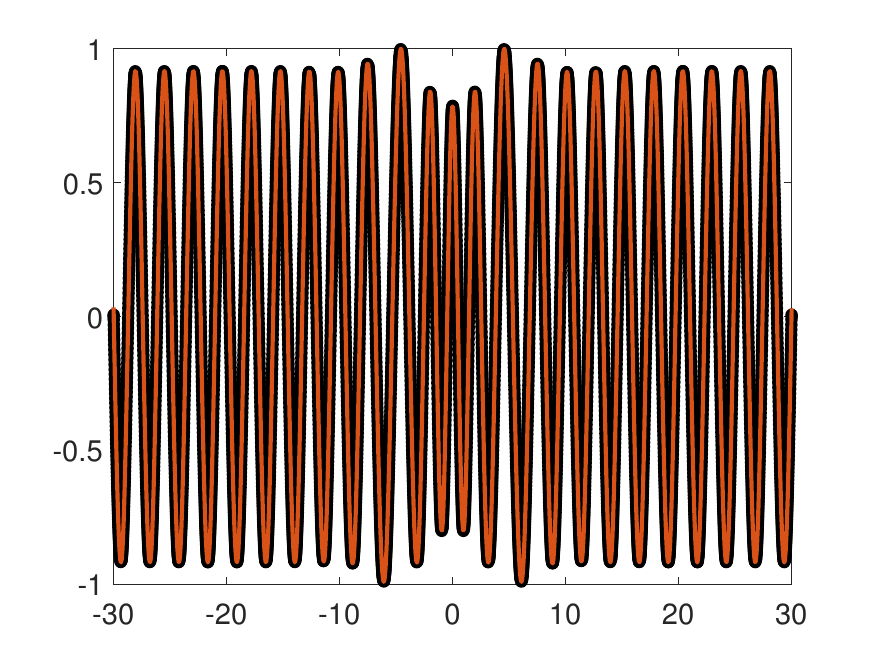}
        \put (50,-1) {$\displaystyle x$}
        \put (47,72) {$\displaystyle \psi_\lambda(x)$}
        \end{overpic}
    \end{minipage}
    \caption{\label{fig:schrodinger_short} Approximations to an even generalized eigenfunction corresponding to $\lambda=8$ (right panel) of a Schr\"odinger operator with a short-range potential $v(x) = -5\cos(x/2)\exp(-x^2/32)$ (left panel), computed using a $4^{\rm th}$-order wave packet scheme with $\epsilon=0.1$ (orange) and a PML scheme (black).}
\end{figure}

\begin{figure}[tbp!]
    \centering
    \begin{minipage}{0.48\textwidth}
        \begin{overpic}[width=\textwidth]{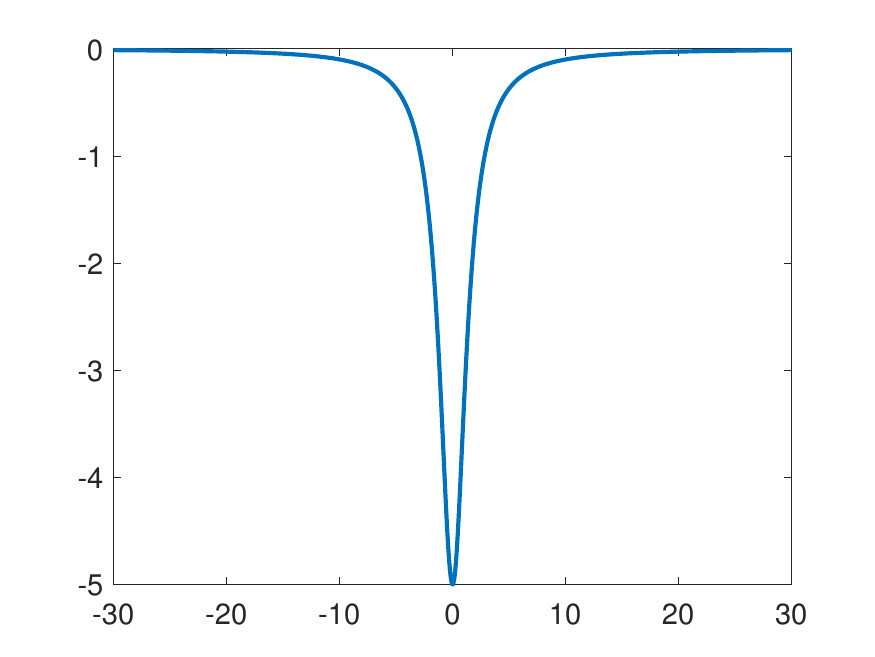}
        \put (50,-1) {$\displaystyle x$}
        \put (47,72) {$\displaystyle v(x)$}
        \end{overpic}
    \end{minipage}
    \begin{minipage}{0.48\textwidth}
        \begin{overpic}[width=\textwidth]{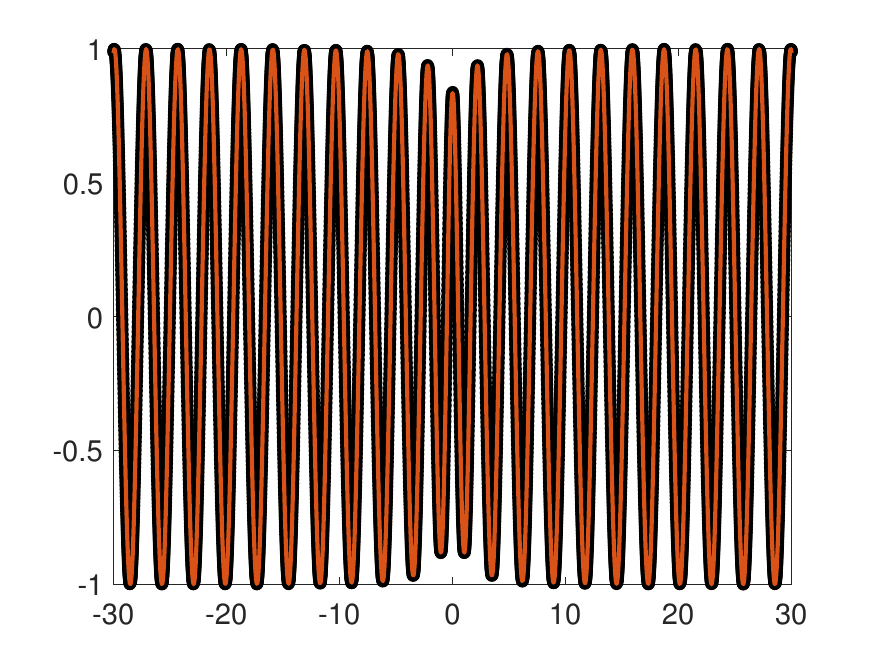}
        \put (50,-1) {$\displaystyle x$}
        \put (47,72) {$\displaystyle \psi_\lambda(x)$}
        \end{overpic}
    \end{minipage}
    \caption{\label{fig:schrodinger_long} Approximations to an even generalized eigenfunction corresponding to $\lambda=5$ (right panel) of a Schr\"odinger operator with a long-range potential $v(x) = -10/(2+x^2)$ (left panel), computed using a $4^{\rm th}$-order wave packet scheme with $\epsilon=0.1$ (orange) and a PML scheme (black).}
\end{figure}

\subsubsection{Example 5: 2D Schr\"odinger operator }

In a two-dimensional planar strip $\Omega = \{(x,y):x\in\mathbb{R},y\in(-1,1)\}$ of unit width with one unbounded dimension, the Laplacian
\begin{equation}\label{eqn:2DLap_strip}
-\Delta u = -\partial_x^2 u - \partial_y^2 u, \qquad\text{where}\qquad u\in H^2(\Omega),
\end{equation}
has an absolutely continuous spectrum on $(\pi^2/4,\infty)$. The generalized eigenfunctions of the Laplacian are tensor products of square-integrable transverse modes and generalized longitudinal modes: given a point $\lambda = (2\pi k)^2 + (n\pi/2)^2$ in the continuous spectrum with real wave number $k\in\mathbb{R}$ and positive integer $n\geq 1$, the generalized eigenfunctions have the form
\begin{equation}
\psi_{k,n}(x,y)=
\begin{cases}
\exp(2\pi i k x)\cos( n\pi y / 2), \qquad &n = 1,3,5,\ldots, \\
\exp(2\pi i k x)\sin( n\pi y / 2), \qquad &n = 2,4,6,\ldots.
\end{cases}
\end{equation}
Note that the multiplicity of the spectrum increases as $\lambda$ increases, beginning with $m=2$ between $\pi^2/4$ and $\pi^2$, $m=4$ between $\pi^2$ and $9\pi^2/4$, and so on, increasing by $2$ at each eigenvalue of the one-dimensional Laplacian along the transverse direction. \Cref{fig:2DLap_regular} shows $4^{\rm th}$-order wave packet approximations to generalized eigenfunctions associated with $\lambda=\pi(\pi-1)$ corresponding to the lowest transverse mode ($n=1$) and with $\lambda=\pi(\pi+1)$ corresponding to the lowest odd transverse mode ($n=2$).

The spectral measures of the two-dimensional Laplacian can be computed analytically via Fourier analysis, analogous to~\cref{sec:diff_ex}, and we find that
$$
\rho_f(\lambda) = \sum_{n = 1}^\infty\ \frac{\left(\hat f\left(\sqrt{\lambda - (n\pi/2)^2}/(2\pi)\right)\right)^2-\left(\hat f\left(-\sqrt{\lambda - (n\pi/2)^2}/(2\pi)\right)\right)^2}{\sqrt{\lambda - (n\pi/2)^2}/(2\pi)}\chi_{\lambda\geq (n\pi/2)^2}.
$$
In general, the spectral measure may have singularities at the points $(n\pi / 2)^2$ where the multiplicity changes due to the transverse modes. \Cref{fig:2DLap_singular} illustrates how the wave-packet approximations may be corrupted at points very close to these singularities in the measure, as nearby transverse modes cause local phase distortions in the wave packet approximation that are only removed slowly as $\epsilon\rightarrow 0$. This is a Fourier analogue of the phenomenon illustrated  in~\cref{fig:mult_op_poisson} for the multiplication operator.

\begin{figure}[tbp!]
    \centering
    \begin{minipage}{0.48\textwidth}
        \begin{overpic}[width=\textwidth]{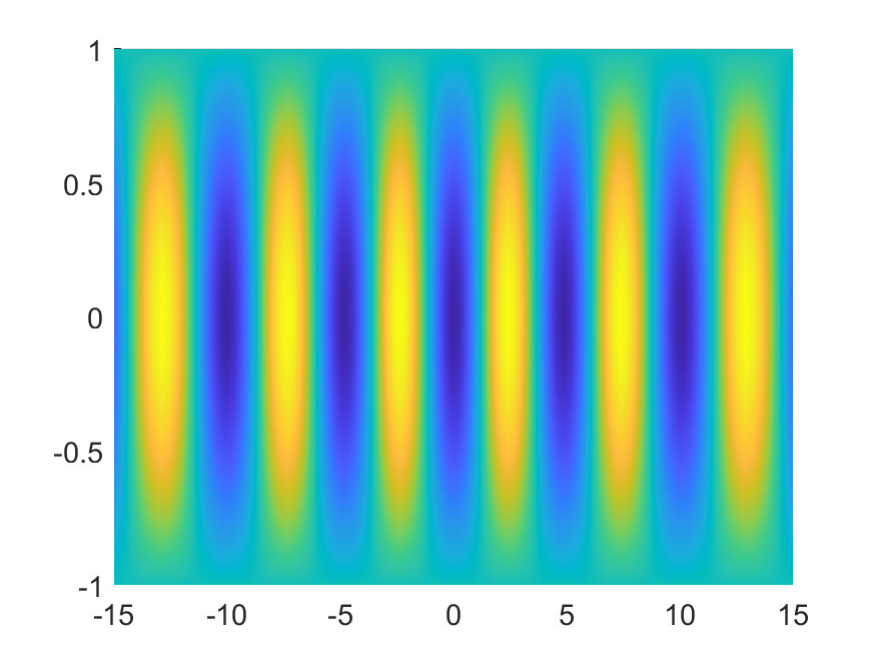}
        \put (51,-1) {$\displaystyle x$}
        \put (2,38) {$\displaystyle y$}
        \end{overpic}
    \end{minipage}
    \begin{minipage}{0.48\textwidth}
        \begin{overpic}[width=\textwidth]{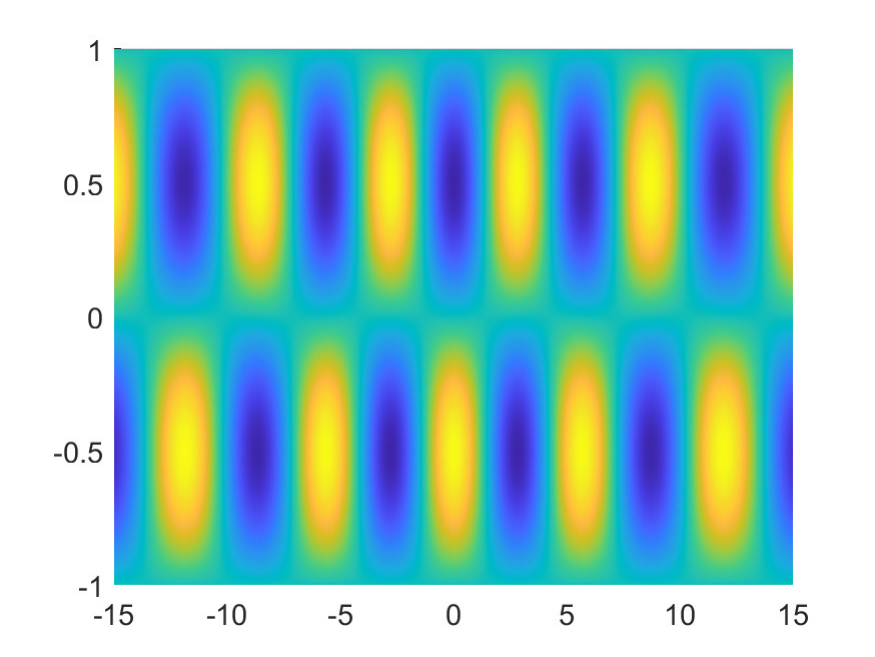}
        \put (51,-1) {$\displaystyle x$}
        \put (2,38) {$\displaystyle y$}
        \end{overpic}
    \end{minipage}
    \caption{\label{fig:2DLap_regular} Wave packet approximations to the generalized eigenfunctions of the two-dimensional Laplacian in a strip of unit width, corresponding to $\lambda=\pi(\pi-1)$ (left panel) and $\lambda=\pi(\pi+1)$ (right panel).}
\end{figure}

\begin{figure}[tbp!]
    \centering
    \begin{minipage}{0.48\textwidth}
        \begin{overpic}[width=\textwidth]{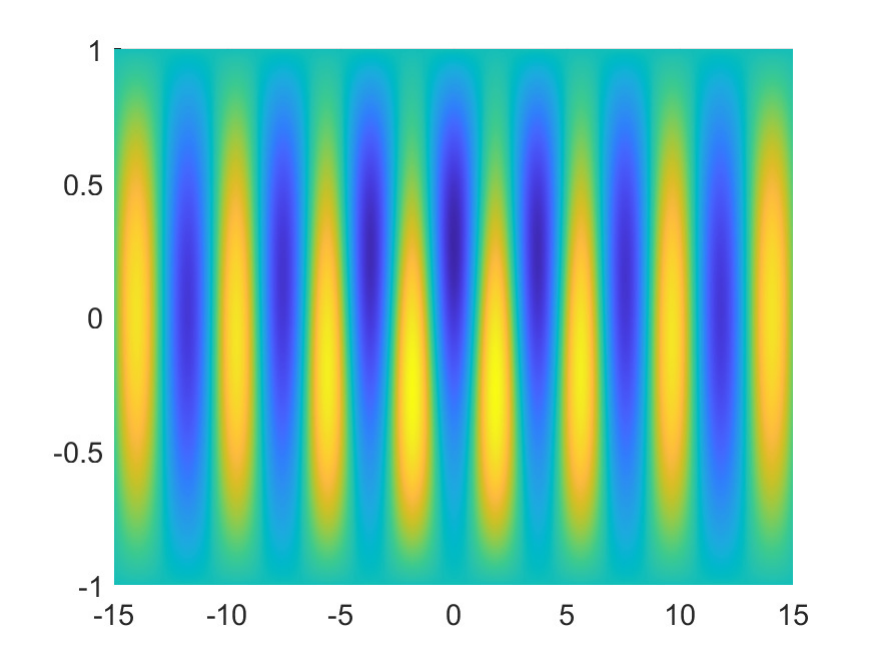}
        \put (51,-1) {$\displaystyle x$}
        \put (2,38) {$\displaystyle y$}
        \end{overpic}
    \end{minipage}
    \begin{minipage}{0.48\textwidth}
        \begin{overpic}[width=\textwidth]{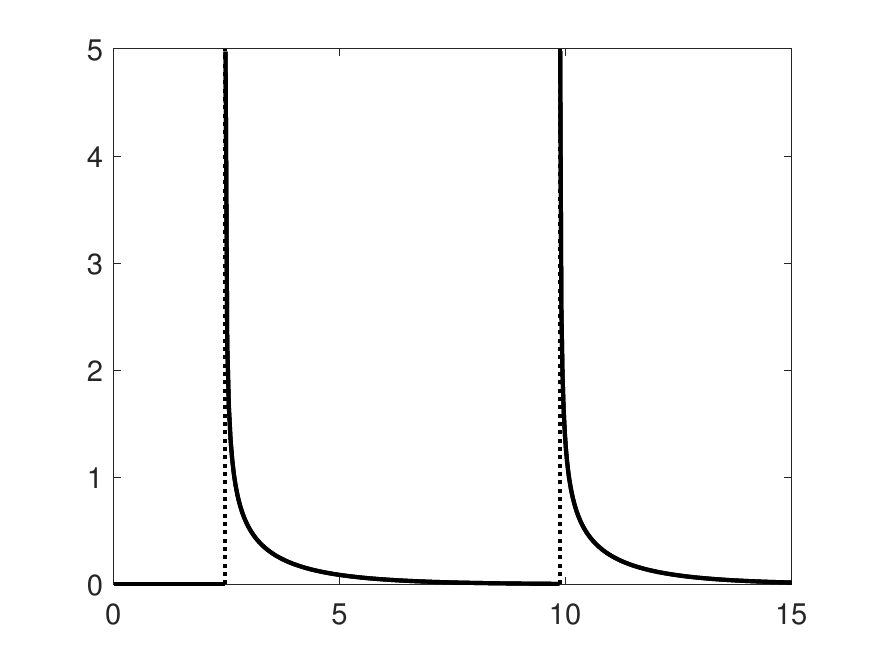}
        \put (50,-1) {$\displaystyle \lambda$}
        \put (47,72) {$\displaystyle \rho_f(\lambda)$}
        \end{overpic}
    \end{minipage}
    \caption{\label{fig:2DLap_singular} A $4^{\rm th}$-order wave packet approximation to a generalized eigenfunction of the two-dimensional Laplacian in a strip of unit width (left panel), corresponding to $\lambda=\pi^2-0.2$, is corrupted with $\epsilon=0.1$ due to unwanted contributions from nearby transverse sinusoidal modes (with spectral parameters $\lambda \geq \pi^2$). The spectral measure (right panel) has singularities (dashed lines) at $\lambda=\pi^2/4$, the leftmost endpoint of the spectrum, and $\lambda=\pi^2$, where the transverse sinusoidal modes appear and the multiplicity of the continuous spectrum jumps from $2$ to $4$.}
\end{figure}

\section{Application to Internal Waves}\label{sec:int_waves}

For a bounded, simply connected open set $\Omega \subset \mathbb{R}^2$, the Poincar\'e problem is the following evolution problem~\cite{sobolev2006new,ralston1973stationary,brouzet2016internal,dauxois2018instabilities,de2020attractors}:
\begin{equation}\label{CHAP4eq:PoincarePDE}
\begin{split}
\left[\frac{\partial^2}{\partial t^2}\left(\frac{\partial^2 g}{\partial x_1^2}+\frac{\partial^2 g}{\partial x_2^2}\right)+\frac{\partial^2 g}{\partial x_2^2}\right](x,t) = u(x) \cos(\lambda t) ,\\
g(x,0)=[\partial_t g](x,0)=0, \quad g\restriction_{\partial \Omega } = 0,
\end{split}
\end{equation}
where $\lambda \in (0, 1)$ and $g$ is the streamfunction. This equation models internal waves in a stratified fluid within a two-dimensional aquarium driven by an oscillatory forcing term. Internal waves are key in oceanography and rotating fluids studies~\cite{maas2005wave,sibgatullin2019internal}. $\Omega$'s geometry and the forcing frequency $\lambda$ can concentrate fluid velocity on attractors, a phenomenon predicted by~\cite{maas1995geometric} and confirmed experimentally by~\cite{maas1997observation}.

Let $-B$ be the Dirichlet Laplacian on $\Omega$, $A=-\partial_{x_2}^2$ and
$$
L=B^{-1}A:H_0^1(\Omega)\rightarrow H_0^1(\Omega).
$$
The operator $L$ generates a wave equation equivalent to~\eqref{CHAP4eq:PoincarePDE} and $\Lambda(L)=[0,1]$~\cite{ralston1973stationary}. Its spectral type, however, is a complicated matter. Descriptions of singular profiles in the long-time evolution of~\cref{CHAP4eq:PoincarePDE}, along with associated spectral results of $L$, are accessible via microlocal analysis~\cite{de2020attractors,colin2020spectral,dyatlov2019microlocal,tao20190,wang2022dynamics,wangscattering,dyatlov2021mathematics}. There is also significant study of $L$ in the Russian literature, where is known as the Sobolev problem. See, for example, \cite{fokin1993existence,troitskaya2017mathematical}. In particular, some domains $\Omega$ lead to singular continuous spectra. It is known that if the forcing frequency meets the Morse--Smale conditions in a neighborhood around $\lambda$, the singular profile relates to \cite[Theorem 2]{dyatlov2021mathematics}
$$
u^{+}=\lim_{\epsilon\downarrow 0}(A-(\lambda^2-i\epsilon)B)^{-1}u\in H^{1/2-}(\Omega).
$$
The imaginary part of $u^{+}$ corresponds to a generalized eigenfunction, which we can compute using our method.

\begin{figure}
\centering
\includegraphics[width=0.49\textwidth,clip,trim={0mm 0mm 0mm 0mm}]{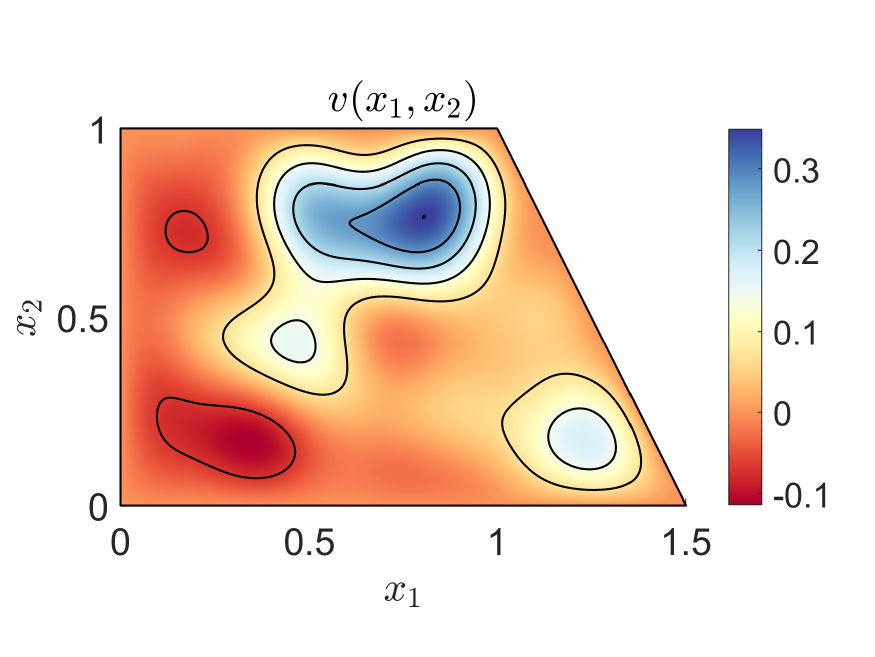}
\hfill
\includegraphics[width=0.49\textwidth,clip,trim={0mm 0mm 0mm 0mm}]{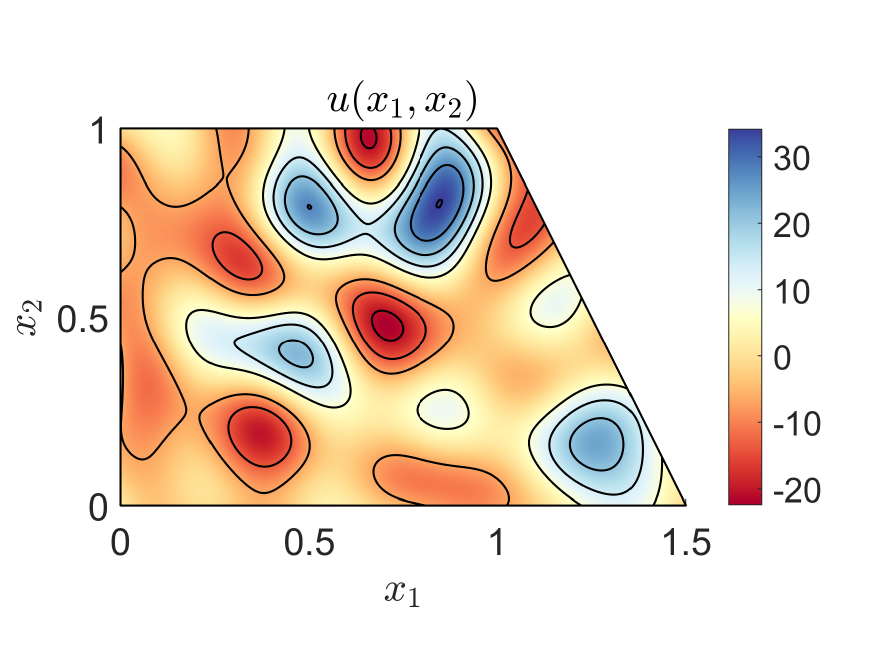}\vspace{-5mm}
\caption{Left: Random function on a trapezium for which we compute spectral measures. Right: The function $u=Bv$.}
\label{fig:internal_wave_u}
\end{figure}

\begin{figure}
\centering
\includegraphics[width=0.6\textwidth,clip,trim={0mm 0mm 0mm 0mm}]{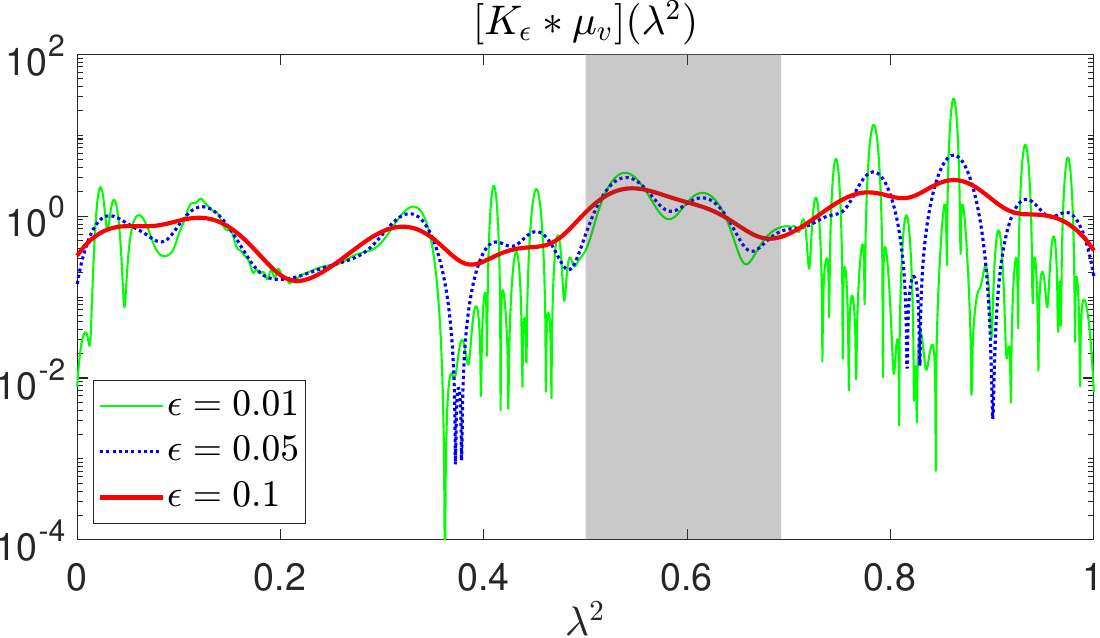}
\caption{Spectral measure of $A-\lambda B$ with respect to the random function $v$ shown in~\cref{fig:internal_wave_u}. The horizontal axis is scaled against frequency, and the spectrum is purely absolutely continuous in the shaded region.}
\label{fig:measure}
\end{figure}

Let $\Omega$ be a trapezium and $v$ a randomly chosen function, as shown in \cref{fig:internal_wave_u}. To solve the linear systems and compute the resolvent, we employ ultraSEM \cite{fortunato2021ultraspherical}, a hp-adaptive finite element code. The resulting smoothed spectral measures for various $\epsilon$ and the $6^{\rm th}$ order kernel from \cref{fig:mth_order_kernel} are presented in \cref{fig:measure}. The shaded area denotes regions where the spectrum is known to be absolutely continuous~\cite{dyatlov2021mathematics}. Here, the spectral measure has converged. There are additional areas with pronounced oscillations, indicating a singular spectrum. \cref{fig:num_attractors} plots the computed functions $u^+$, where the spectrum is locally absolutely continuous around each selected value of $\lambda$. These are compared to the internal wave attractors observed experimentally in \cite{hazewinkel2010observations}, shown in \cref{fig:exp_attractors}. The similarity between these numerical computations and the experimental observations is striking.

\begin{figure}[th!]
\centering
\begin{tabular}{c|c|c|c}
$\lambda=0.550$&$\lambda=0.590$&$\lambda=0.686$&$\lambda=0.790$\\
\hline
\includegraphics[width=0.22\textwidth,clip,trim={0mm -6mm 0mm -6mm}]{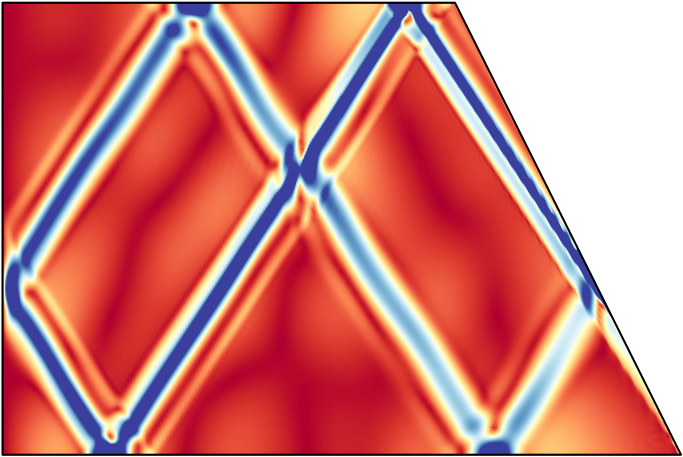} & 
\includegraphics[width=0.22\textwidth,clip,trim={0mm -6mm 0mm -6mm}]{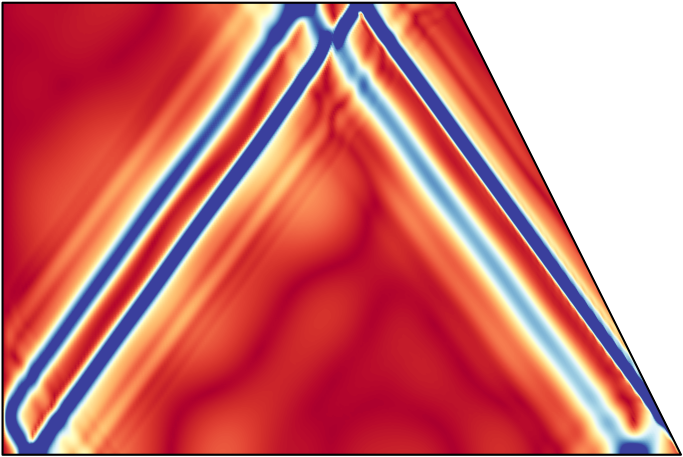} &
\includegraphics[width=0.22\textwidth,clip,trim={0mm -6mm 0mm -6mm}]{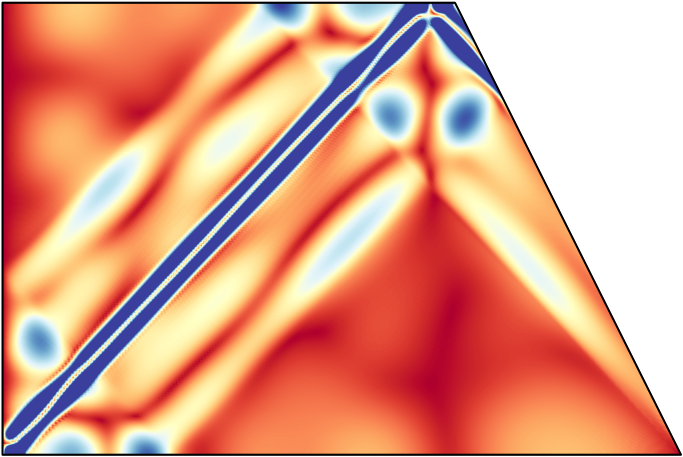} &
\includegraphics[width=0.22\textwidth,clip,trim={0mm -6mm 0mm -6mm}]{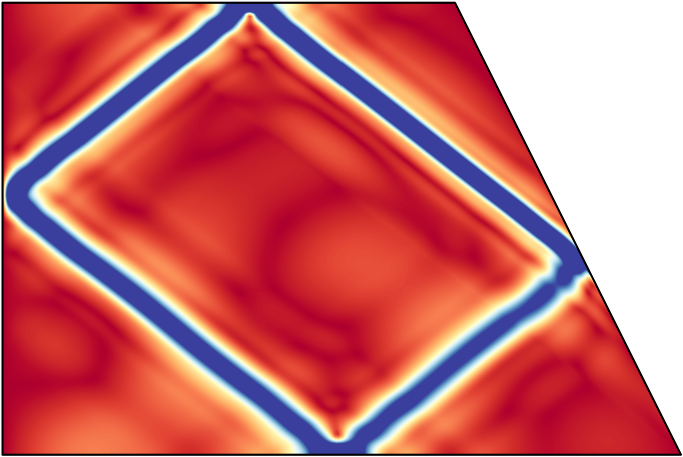}\\
\hline
\includegraphics[width=0.22\textwidth,clip,trim={0mm -6mm 0mm -6mm}]{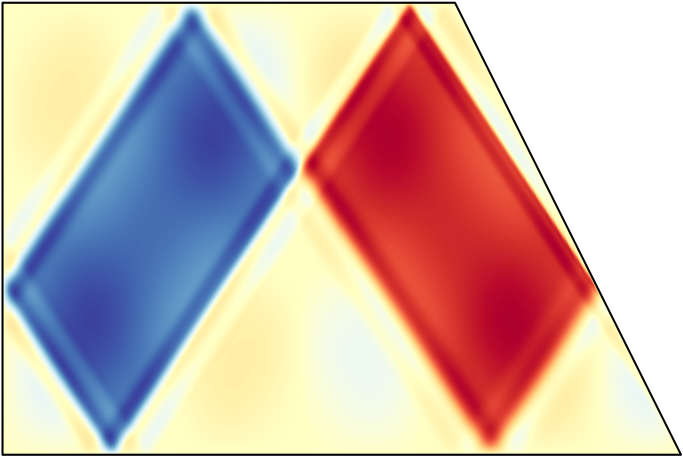} & 
\includegraphics[width=0.22\textwidth,clip,trim={0mm -6mm 0mm -6mm}]{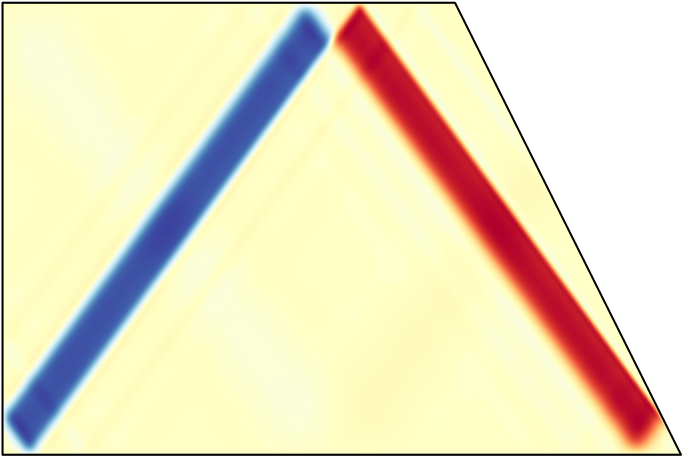} &
\includegraphics[width=0.22\textwidth,clip,trim={0mm -6mm 0mm -6mm}]{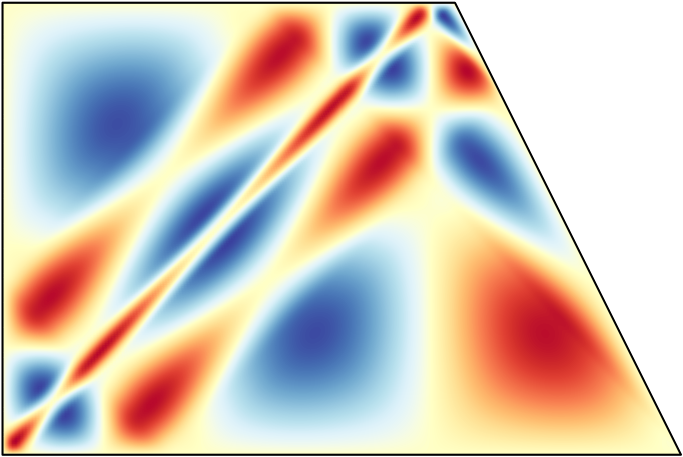} &
\includegraphics[width=0.22\textwidth,clip,trim={0mm -6mm 0mm -6mm}]{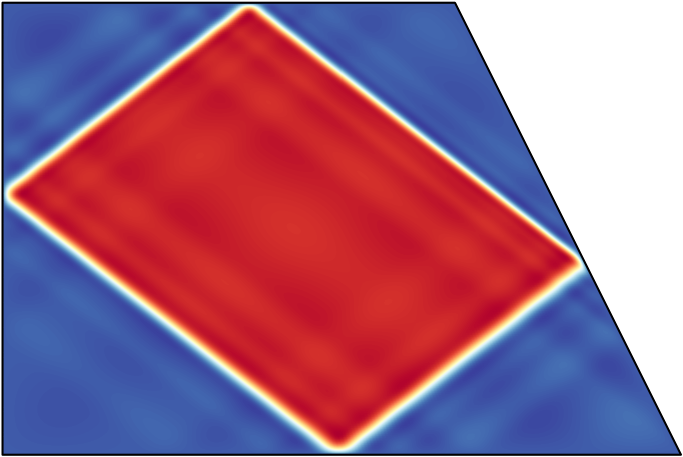}\\
\end{tabular}
\caption{\label{fig:num_attractors}The function $u^+$ approximated using $\epsilon=0.01$ and the $6^{\rm th}$ order kernel. The top row shows the absolute value of $\partial_{x_1}u^+$, corresponding to the fluid velocity concentration on the attractor. The bottom part shows the imaginary part of $u^+$, which corresponds to a generalized eigenfunction. These plots show striking similarity with the experimentally observed attractors in~\cref{fig:exp_attractors}.}
\vspace{2mm}
\begin{tabular}{c|c|c|c}
\includegraphics[width=0.22\textwidth,clip,trim={0mm -2mm 0mm -2mm}]{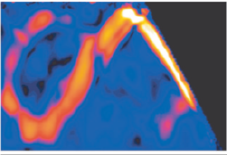} & 
\includegraphics[width=0.22\textwidth,clip,trim={0mm -2mm 0mm -2mm}]{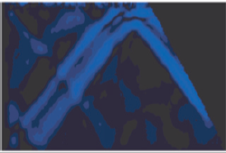} &
\includegraphics[width=0.22\textwidth,clip,trim={0mm -2mm 0mm -2mm}]{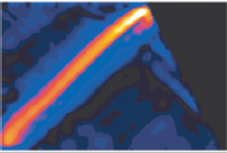} &
\includegraphics[width=0.22\textwidth,clip,trim={0mm -2mm 0mm -2mm}]{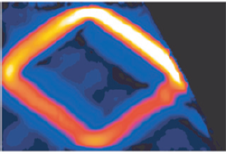}\\
\hline
\includegraphics[width=0.22\textwidth,clip,trim={0mm -2mm 0mm -2mm}]{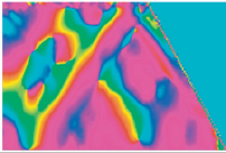} & 
\includegraphics[width=0.22\textwidth,clip,trim={0mm -2mm 0mm -2mm}]{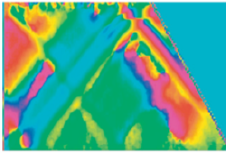} &
\includegraphics[width=0.22\textwidth,clip,trim={0mm -2mm 0mm -2mm}]{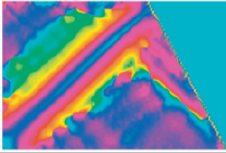} &
\includegraphics[width=0.22\textwidth,clip,trim={0mm -2mm 0mm -2mm}]{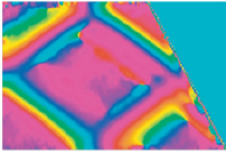}\\
\end{tabular}
\caption{\label{fig:exp_attractors}From~\cite{hazewinkel2010observations}: Observations of steady-state patterns in terms of harmonic amplitude (top) and phase (bottom). Reproduced from \cite{hazewinkel2010observations}, with the permission of AIP Publishing.}
\end{figure}

\appendix

\section{Proofs of convergence and approximation rates}\label{sec:appendix}

In this section, we prove the convergence results and approximation rates stated in~\cref{sec:convergence}. For clarity, we restate some results before their proof. We begin with an approximation result for complex Borel measures smoothed by rational convolution kernels (c.f.,~\cref{sec:rat_kern}). Similar convergence rates for smooth measures were derived by Colbrook et al~\cite{colbrook2020computing}. We derive a sharper result for non-smooth measures, proving convergence on the Lebesgue set of the Radon-Nikodym derivative.

Recall that the Lebesgue set of a (possibly complex-valued) integrable function $\rho:\mathbb{R}\rightarrow\mathbb{C}$ is the set of Lebesgue points,  given by
$$
\Gamma_{\rm Leb}(\rho) = \left\{\lambda\in\mathbb{R} : \lim_{\epsilon\rightarrow 0}\frac{1}{2\epsilon}\int_{\lambda-\epsilon}^{\lambda+\epsilon}|\rho(\tilde\lambda)-\rho(\lambda)|\dd\tilde\lambda=0\right\}.
$$
The Lebesgue set of an integrable function has full measure and every point of continuity is a Lebesgue point~\cite[p.~106]{stein2009real}.

\begin{lemma}\label{lem:borel_meas_approx}
Given $\lambda\in\mathbb{R}$ and $\delta>0$, let $\mu$ be a Borel measure on $\mathbb{R}$ that is absolutely continuous on the closed interval $\Omega=[\lambda+\delta,\lambda+\delta]$ with Radon-Nikodym derivative $\rho$, and let $K$ be an $m$th order kernel for some integer $m\geq 1$. It holds that
$$
\lim_{\epsilon\downarrow0}|[K_\epsilon * \mu](\lambda)-\rho(\lambda)| = 0 \qquad\forall \lambda\in\Gamma_{\rm Leb}(\rho).
$$
Furthermore, let $0<\epsilon<1$ and the total variation norm of $\mu$ be $\|\mu\|$. If $\rho\in C^{n,\alpha}(\Omega)$ for some $0\leq\alpha\leq 1$ and integer $n\geq 1$, then
$$
|[K_\epsilon * \mu](\lambda)-\rho(\lambda)| \leq \frac{\epsilon^m \|\mu\|}{(\epsilon+\delta)^{m+1}} +  \begin{cases}
\|\rho\|_{C^{n,\alpha}(\Omega)} (1+\delta^{-n-\alpha})\epsilon^{n+\alpha}, \qquad& n+\alpha < m, \\
\|\rho\|_{C^m(\Omega)} (1+\delta^{-m})\epsilon^m\log(1+\epsilon^{-1}), \qquad& n+\alpha\geq m.
\end{cases}
$$
\end{lemma}

\begin{proof}
Using the fact that $\mu$ is absolutely continuous on $\Omega$, we partition the convolution of the measure and calculate that
\begin{equation}\label{eqn:approx_error}
|[K_\epsilon * \mu](\lambda)-\rho(\lambda)| \leq\left|\int_\Omega K_\epsilon(\lambda-\tilde\lambda)\rho(\tilde\lambda)\dd\lambda - \rho(\lambda)\right|+\left|\int_{\mathbb{R}\setminus\Omega} K_\epsilon(\lambda-\tilde\lambda)\dd\mu(\tilde\lambda)\right|.
\end{equation}
The first term on the right-hand side of~\eqref{eqn:approx_error} vanishes as $\epsilon\rightarrow 0$ when $\lambda\in\Omega$ is a Lebesgue point of $\rho$. This follows from Theorem 1.25 in the textbook by Stein and Weiss~\cite[p.~13]{stein1971introduction}, because $K$ and $\chi_\Omega\rho$ are integrable and the decay condition in~\cref{def:mth_order_kernel} implies that $\sup_{t\geq\lambda}|K(t)|$ is integrable with respect to $\lambda$. The second term also vanishes in the limit $\epsilon\rightarrow 0$, which follows from the decay condition for $K$ in~\cref{def:mth_order_kernel} and the fact that $\smash{\mu_{f,\phi}(\Delta_{-\infty}^\lambda)}$ is a function (with respect to $\lambda$) of bounded variation so that 
\begin{equation}\label{eqn:decay_bound}
\left|\int_{\mathbb{R}\setminus\Omega} K_\epsilon(\lambda-\tilde\lambda)\dd\mu(\tilde\lambda)\right|\leq \frac{C_k\epsilon^m}{[\inf_{\tilde\lambda\in\mathbb{R}\setminus\Omega}|\lambda-\tilde\lambda|]^{m+1}+\epsilon^{m+1}}\int_{\mathbb{R}\setminus\Omega}\dd|\mu|(\tilde\lambda)\rightarrow 0, \qquad\text{as } \epsilon\rightarrow 0,
\end{equation}
where $|\mu|$ is the total variation of $\mu$. This establishes convergence at the Lebesgue points of $\rho$ as $\epsilon\downarrow 0$. 

The derivation of the convergence rates for H\"older continuous densities also uses the partitioned convolution in~\cref{eqn:approx_error}. The first term in the bound in the statement of the theorem follows directly from~\cref{eqn:decay_bound}. The second term in the bound can be computed by expanding the integrand of the first term in~\cref{eqn:approx_error} in a Taylor series about $\lambda$, applying the vanishing moment conditions on $K$, and carefully bounding the remainder. The calculations are identical to the detailed proof of Theorem~5.2 in Appendix A.1 of Colbrook et al~\cite{colbrook2020computing}, with the exception that \cite{colbrook2020computing} focus on probability measures, in which case $\|\mu\|=1$.
\end{proof}

The weak convergence of wave packets in~\cref{thm:conv_ae,thm:conv_rates} are corollaries of~\cref{lem:borel_meas_approx}.

\begin{proof}[Proof of Theorems 1 and 2]
In light of~\cref{lem:borel_meas_approx}, consider the wave packet approximation, $\smash{u_\lambda^{(\epsilon)}}$, defined in~\cref{eqn:mth_stone}. If the spectrum of $A$ is absolutely continuous on an interval $\Omega$, then the spectral measures $\mu_{f,\phi}$ are also absolutely continuous there with Radon-Nikodym derivatives $\rho_{f,\phi}$. By the chain rule, we calculate that (for almost every $\lambda\in\Omega$)
\begin{equation}\label{eqn:gef_meas1}
\rho_f(\lambda)\langle u_\lambda| \phi\rangle = \rho_f(\lambda)\frac{\mathrm{d}\mu_{f,\phi}}{\mathrm{d}\mu_f}(\lambda) = \rho_{f,\phi}(\lambda) \qquad \forall\phi\in\Phi.
\end{equation}
On the other hand, from the spectral theorem for self-adjoint operators on a Hilbert space, we know that $\smash{\langle u^{(\epsilon)}_\lambda , \phi\rangle} = [K_\epsilon * \mu_{f,\phi}](\lambda)$. Consequently,~\cref{lem:borel_meas_approx} implies that $\langle \smash{u^{(\epsilon)}_\lambda , \phi\rangle \rightarrow \rho_f(\lambda)\langle u_\lambda| \phi\rangle}$ at every point in the Lebesgue set of $\rho_{f,\phi}$ (\cref{thm:conv_ae}) and establishes convergence rates for H\"older continuous densities (\cref{thm:conv_rates}). 
\end{proof}

To derive appropriate conditions under which the wave packet approximations converge in finer topologies (c.f.~\cref{thm:conv_point,thm:conv_abs}), we carefully reformulate the wave packet identity in~\cref{eqn:wave_packet} in terms of generalized eigenfunctions and extend it to the relevant topological vector space (TVS). We employ a weak notion of integrability of a vector-valued map $u:S \rightarrow V$ with respect to a Borel measure $\mu$ on a Hausdorff space $S$, where $V$ is a locally convex TVS whose dual $V^*$ separates points in $V$. The map $u$ is called Pettis integrable if, for each measurable set $\Omega\in\mathcal{B}(S)$, there is a $u_\Omega\in V$ such that~\cite{thomas1975integration}
\begin{equation}\label{eqn:pettis_integral}
\langle \ell | u_\Omega\rangle = \int_\Omega \langle \ell | u(\lambda) \rangle \dd\mu(\lambda), \qquad\text{for all}\qquad \ell\in V^*.
\end{equation}
The element $u_\Omega$ is called the Pettis integral of $u$ over $\Omega$ and we write $u_\Omega=\smash{\int_\Omega u\,\dd\mu(\lambda)}$.

To begin, we show that the wave packet identity for $u_\lambda^{(\epsilon)}$ can be understood via the Pettis integral of the map $\tilde\lambda\rightarrow K_\epsilon(\lambda-\tilde\lambda)u_{\tilde\lambda}$ in the locally convex TVS $\Phi^*$.

\begin{lemma}\label{lem:wp_expansion}
Let $A:D(A)\rightarrow H$ be a selfadjoint operator on a rigged Hilbert space $\Phi\subset H\subset\Phi^*$, where $\Phi$ is a countably Hilbert nuclear space and $A\Phi\subset\Phi$. Let $K$ be an $m$th order rational kernel satisfying $(i)-(iii)$ in~\cref{def:mth_order_kernel} and~\cref{eqn:vandermonde_condition}, and, given $f\in\Phi$, let $\smash{u_\lambda^{(\epsilon)}}$ be the associated wave packet approximation in~\cref{eqn:mth_stone}. If $\{u_\lambda\}_{\lambda\in\mathbb{R}}\subset\Phi^*$ are the generalized eigenfunctions of $A$ defined in~\cref{eqn:pvm_to_geig}, then for each fixed $\epsilon>0$ and $\lambda\in\mathbb{R}$, the map $\tilde \lambda \rightarrow K_\epsilon(\lambda-\tilde\lambda)u_{\tilde\lambda}$ is $\mu_f$-integrable with Pettis integrals
$$
E(\Omega)u_\lambda^{(\epsilon)} = \int_\Omega K_\epsilon(\lambda-\tilde\lambda)u_{\tilde\lambda}\,\dd\mu_f(\tilde\lambda)\in\Phi^*, \qquad\text{for each}\qquad \Omega\in\mathcal{B}(\mathbb{R}).
$$
Moreover, each such Pettis integral lies in the image $i(H)\subset\Phi^*$, where $i:H\hookrightarrow\Phi^*$ is the continuous injection of $H$ into $\Phi^*$ and, consequently, the equality holds in $H$.
\end{lemma}

\begin{proof}
Given any Borel set $\Omega\subset\mathcal{B}(\mathbb{R})$ and $\lambda\in\mathbb{R}$, we calculate directly that
\begin{equation*}
\langle E(\Omega)u_\lambda^{(\epsilon)} | \phi \rangle = \langle E(\Omega)u_\lambda^{(\epsilon)} , \phi \rangle 
= \int_\Omega K_\epsilon(\lambda-\tilde\lambda)\,\dd\mu_{f,\phi}(\tilde\lambda)
= \int_\Omega K_\epsilon(\lambda-\tilde\lambda)\langle u_{\tilde\lambda} | \phi\rangle \,\dd\mu_f(\tilde\lambda).
\end{equation*}
The first equality follows from the compatibility of the inner product on $H$ with the dual pairing on $\Phi^*\times\Phi$. The third equality follows from~\cref{eqn:geig_action}, i.e., that $\langle u_\lambda | \phi\rangle$ is the Radon-Nikodym derivative of $\mu_{f,\phi}$ with respect to $\mu_f$. For every $\tilde\lambda\in\mathbb{R}$, we have that $K_\epsilon(\lambda-\tilde\lambda)\langle u_{\tilde\lambda} | \phi\rangle =\langle  K_\epsilon(\lambda-\tilde\lambda) u_{\tilde\lambda} | \phi\rangle$ by linearity of the dual pairing. For each $\Omega\in\mathcal{B}(\mathbb{R})$, the above calculation holds for every $\phi\in\Phi$, which shows that $\tilde\lambda\rightarrow K_\epsilon(\lambda-\tilde\lambda)u_{\tilde\lambda}\in\Phi^*$ is Pettis integrable with Pettis integral over $\Omega$ given by $\smash{E(\Omega)u_\lambda^{(\epsilon)}\in\Phi^*}$. By the compatibility between the dual pairing on $\Phi^*\times\Phi$ and the inner product on $H$, $\smash{E(\Omega)u_\lambda^{(\epsilon)}\in\Phi^*}$ clearly defines a bounded linear function on $H$, corresponding to $\smash{E(\Omega)u_\lambda^{(\epsilon)}\in H}$ under the canonical identification of $H$ with its dual. Therefore, $\smash{E(\Omega)u_\lambda^{(\epsilon)}\in i(H)\subset\Phi^*}$.
\end{proof}

To extend the wave packet identity to other locally convex TVS and establish the convergence claims in~\cref{thm:conv_point,thm:conv_abs}, we will need to demonstrate Pettis integrability of the map $\tilde\lambda\rightarrow K_\epsilon(\lambda-\tilde\lambda)u_{\tilde\lambda}$ in other spaces. For this task,  we collect useful criteria for Pettis integrability on Suslin spaces, due to Thomas~\cite{thomas1975integration}, in the next lemma.

\begin{lemma}\label{lem:pettis_integrability_criteria}
Given a locally convex Suslin TVS $V$ and a Hausdorff space $S$ with finite Borel measure $\mu$, suppose that $u:S\rightarrow V$ has $u(S)\subset V$ bounded and that, for some subset $W\subset V^*$ separating points on $V$, $\langle \ell | u(\cdot)\rangle$ is $\mu$-measurable for all $\ell\in W$. Then, $u$ is Pettis integrable with respect to $\mu$ in the sense of~\cref{eqn:pettis_integral}.
\end{lemma}
\begin{proof}
Since $u$ is weakly measurable with respect to the total set $W\subset V^*$, $u$ is measurable~\cite[Thm.~1]{thomas1975integration}. Since $u(S)\subset V$ is also bounded, it is Pettis integrable (c.f.~\cite[Cor.~3.1]{thomas1975integration} and~\cite[Thm.~3]{thomas1975integration}).
\end{proof}

We now prove~\cref{thm:conv_point}, i.e., that when $H=L^2(S,\nu)$ and the generalized eigenfunctions satisfy suitable uniform continuity and integrability assumptions, the wave packet approximations converge uniformly on compact subsets of $S$. The basic idea is to show that the wave packet identity holds pointwise on the relevant compact set $K$ under the hypothesis of~\cref{thm:conv_point} and apply~\cref{lem:borel_meas_approx} to each $x$-slice. The $K$-uniformity hypotheses ensure that convergence and convergence rates are uniform over $K$. (We can also replace the condition $\sup_{\lambda}\|u_\lambda|_K\|_{C(K)}<\infty$ with the weaker condition that $\|u_\lambda|_K\|_{C(K)}$ is $\mu_f$-integrable here if we use Bochner integrability, the separability of $C(K)$, and the Pettis measurability theorem.)

\vspace{2mm}
\noindent{\textbf{Restatement of \cref{thm:conv_point}}} (Uniform convergence on compact sets)\textbf{.}
\textit{Let the hypotheses of~\cref{thm:conv_ae} hold with $H=L^2(S,\nu)$ for metric space $S$ and strictly positive $\nu$. Given compact domain $K\subset S$, suppose $\Phi|_K$ is dense in $L^2(K,\nu)$ and $u_\lambda$ is continuous on $K$ for all $\lambda\in\mathbb{R}$ with $\sup_{\lambda\in\mathbb{R}}\|u_\lambda|_K\|_{C(K)}<\infty$. Given $\lambda_*\in{\rm int}(\Omega)$ and $\delta>0$ such that $I=[\lambda_*-\delta,\lambda_*+\delta]\in{\rm int}(\Omega)$, if $(x,\lambda) \rightarrow \rho_f(\lambda) u_\lambda|_k(x)$ is $K$-uniformly continuous on $I$ and $\smash{u_{\lambda_*}^{(\epsilon)}|_K\in C(K)}$,
$$
\lim_{\epsilon\rightarrow 0^+} \sup_{x\in K}\lvert u_{\lambda_*}^{(\epsilon)}(x) - \rho_f(\lambda_*)u_{\lambda_*}(x)\rvert = 0.
$$
Moreover, if $(x,\lambda)\rightarrow \rho_f(\lambda)u_\lambda|_K(x)$ is $K$-uniformly $(n,\alpha)$-H\"older continuous on $I$ with nonnegative integer $n$ and $\alpha\in[0,1]$, then
$$
\sup_{x\in K}\lvert u_{\lambda_*}^{(\epsilon)}(x) - \rho_f(\lambda_*)u_{\lambda_*}(x)\rvert = \mathcal{O}(\epsilon^{n+\alpha}) + \mathcal{O}(\epsilon^m\log(\epsilon)),\qquad\text{as}\qquad\epsilon\rightarrow 0^+.
$$}
\vspace{1mm}

\begin{proof}
First, for each $\epsilon>0$, the map $\smash{v_\epsilon:\lambda\rightarrow K_\epsilon(\lambda_*-\lambda)u_\lambda|_K}$ satisfies the criteria for Pettis integrability in~\cref{lem:pettis_integrability_criteria}. We have $\mu_f(\mathbb{R})<\infty$ and, for each $\epsilon>0$, $\smash{v_\epsilon(\mathbb{R})}$ is a bounded set in $C(K)$ because $K_\epsilon(\lambda_*-\cdot)$ is bounded and $\sup_{\lambda\in\mathbb{R}}\|u_\lambda|_K\|_{C(K)}<\infty$. 
To verify measurability, note that each $\phi\in\Phi|_K$ defines a linear functional on $C(K)$ via $\smash{\int_K \phi(x)f(x)\,\dd\nu(x)}$. Moreover, the collection of all such linear functionals separates points in $C(K)$ because $\Phi|_K$ is dense in $L^2(K)$ and $L^2(K)$ separates points in $C(K)$. Now, $\int_k \phi v_\epsilon(\cdot)\,\dd\nu = \langle v_\epsilon(\cdot) | \phi \rangle$ for any $\phi\in\Phi|_K$ because $u_\lambda$ is continuous on $K$. Since $\langle v_\epsilon(\cdot) |\phi \rangle$ is $\mu_f$-integrable by~\cref{lem:wp_expansion} and, in particular, $\mu_f$-measurable, we can apply~\cref{lem:pettis_integrability_criteria} to conclude that $v_\epsilon$ is Pettis integrable on $C(K)$.

Now, we claim that the Pettis integral of $v_\epsilon$ over $\mathbb{R}$ is precisely $\smash{u_{\lambda_*}^{(\epsilon)}|_K}$. For each $\phi\in \Phi$ with compact support in $K$, we have that
\begin{align*}
\int_K \phi \left[\int_\mathbb{R}  v_\epsilon(\lambda)\,\dd\mu_f(\lambda)\right]\,\dd\nu 
&= \int_\mathbb{R} \int_K \phi\,v_\epsilon(\lambda)\,\dd\nu\,\dd\mu_f(\lambda) \\
&= \int_\mathbb{R}  \langle v_\epsilon(\lambda) | \phi\rangle\,\dd\mu_f(\lambda) \\
&= \int_K \phi\, u_{\lambda_*}^{(\epsilon)}\,\dd\nu.
\end{align*}
The first equality follows from Pettis integrability of $\smash{v_\epsilon}$ on $C(K)$ and the second equality follows from the continuity of $u_\lambda$ on $K$. The last inequality follows from~\cref{lem:wp_expansion}, the compatibility of the dual pairing for $\Phi^*$ and $\Phi$ with the inner product on $H$, and the fact that $\Phi$ has compact support in $K$ so that $\langle u_\lambda^{(\epsilon)} | \phi\rangle = \int_K\phi u_\lambda^{(\epsilon)}\,\dd\nu$. Since equality holds for any $\phi\in \Phi$ with compact support in $K$ and $\Phi|_K$ is dense in $L^2(K,\nu)$, we conclude that $\smash{u_{\lambda_*}^{(\epsilon)}}$ is equal to the Pettis integral of $\smash{v_\epsilon}$ over $\mathbb{R}$ at $\nu$-a.e. $x\in K$. On the other hand both $\smash{u_{\lambda_*}^{(\epsilon)}}$ and the Pettis integral of $\smash{v_\epsilon}$ over $\mathbb{R}$ are continuous functions on $K$. Since $\nu$ is strictly positive, every open subset of $K=\overline{D}$ has positive measure and, consequently, $\smash{u_{\lambda_*}^{(\epsilon)}}$ must be equal to the Pettis integral of $\smash{v_\epsilon}$ over $\mathbb{R}$ at every $x\in K$.

Now, denote the linear functional that evaluates a function in $C(K)$ at the point $x$ by $\delta_x$. By the definition of the Pettis integral in~\cref{eqn:pettis_integral}, we have that
$$
u_{\lambda_*}^{(\epsilon)}(x)
= \langle \delta_x | u_{\lambda_*}^{(\epsilon)}\rangle 
= \int_\mathbb{R} K(\lambda_*-\lambda)\langle \delta_x | u_\lambda\rangle \,\dd\mu_f(\lambda) 
= \int_\mathbb{R} K(\lambda_*-\lambda)u_\lambda(x) \,\dd\mu_f(\lambda).
$$
Moreover, $\omega_x(\Omega):\Omega\rightarrow \smash{\int_\Omega K(\lambda_*-\lambda)u_\lambda(x) \,\dd\mu_f(\lambda)}$ defines a finite Borel measure on $\mathcal{B}(\mathbb{R})$ for each $x\in K$. We can therefore apply~\cref{lem:borel_meas_approx} to each $\omega_x$ to obtain convergence, $u_{\lambda_*}^{(\epsilon)}\rightarrow \rho_f(\lambda_*)u_{\lambda_*}$, at each $x\in K$. In fact, since the map $\lambda\rightarrow\rho_f(\lambda)u_\lambda$ is $K$-uniformly continuous, it is easy to check that the first term in the approximation error in~\cref{eqn:approx_error} converges uniformly for all $x\in K$ as $\epsilon\rightarrow 0$. Moreover, since $\lambda\rightarrow u_\lambda$ is bounded in $C(K)$, it is also $K$-uniformly integrable with respect to $\mu_f$ so that the total variation of the Borel measures $\omega_x$ is uniformly bounded on $K$. Therefore the second term in the approximation error in~\cref{eqn:approx_error} (c.f.,~\cref{eqn:decay_bound}) also converges to zero uniformly for all $x\in K$ as $\epsilon\rightarrow 0$. Similarly, when $\lambda\rightarrow \rho_f(\lambda)u_\lambda$ satisfies the $K$-uniformly H\"older conditions, the convergence rates for the measures $\omega_x$ are uniform for all $x\in K$ as $\epsilon\rightarrow 0$.
\end{proof}

To prove~\cref{thm:conv_abs}, i.e., convergence and approximation rates in a locally convex Suslin TVS $F$, we use a strategy similar to the proof of~\cref{thm:conv_point}. First, we show that the wave packet $u_\lambda^{(\epsilon)}$ has the expected generalized eigenfunction expansion in the sense of a Pettis integral on $F$. With the wave packet identity in place, we apply~\cref{lem:borel_meas_approx} to obtain convergence and approximation rates under suitable regularity conditions on $\lambda\rightarrow \rho_f(\lambda)u_\lambda$.

\vspace{2mm}
\noindent{\textbf{Restatement of \cref{thm:conv_abs}}} (Convergence in Suslin spaces)\textbf{.}
\textit{Let the hypotheses of~\cref{thm:conv_ae} hold and, given a locally convex Suslin TVS, $F$, suppose that $\lambda\rightarrow u_\lambda$ has bounded image in $F$. Additionally, suppose that there is a continuous embedding $i:\Phi\rightarrow F^*$ such that $\langle i(\phi) | f \rangle = \langle f | \phi\rangle$ when $f\in F \bigcap \Phi^*$, $\phi\in\Phi$, and that $i(\Phi)\subset F^*$ separates points in $F$. Given any $\lambda_*\in{\rm int}(\Omega)$ and $\psi\in F^*$, if $\smash{u_{\lambda_*}^{(\epsilon)}\in F}$ and $\lambda_*\in\Gamma_{\rm Leb}(\rho_f(\lambda)\langle\psi | u_\lambda\rangle)$, it holds that
$$
\lim_{\epsilon\rightarrow 0^+}\langle \psi | u_{\lambda_*}^{(\epsilon)}\rangle = \rho_f(\lambda_*)\langle \psi | u_{\lambda_*}\rangle.
$$
Moreover, given $\delta>0$ such that $I=[\lambda_*-\delta,\lambda_*+\delta]\in{\rm int}(\Omega)$, if $\lambda\rightarrow \rho_f(\lambda)\langle \psi | u_{\lambda}\rangle$ is $(n,\alpha)$-H\"older continuous on $I$ with nonnegative integer $n$ and $\alpha\in[0,1]$, then
$$
\lvert\langle \psi | u_{\lambda_*}^{(\epsilon)}\rangle - \rho_f(\lambda_*)\langle \psi | u_{\lambda_*}\rangle\rvert = \mathcal{O}(\epsilon^{n+\alpha}) + \mathcal{O}(\epsilon^m\log(\epsilon)),\qquad\text{as}\qquad\epsilon\rightarrow 0^+.
$$
Finally, the conclusions hold if $\langle \psi|\cdot\rangle$ is replaced by any continuous, $\mu_f$-integrable seminorm $|\cdot|$ on $F$.}
\vspace{1mm}

\begin{proof}
First, we claim that, for each fixed $\epsilon>0$, $\lambda\rightarrow K_\epsilon(\lambda_*-\lambda)u_\lambda$ is Pettis integrable on $F$. Since $\lambda\rightarrow u_\lambda$ has bounded image in $F$ and $K_\epsilon(\lambda_*-\cdot)$ is a continuous and bounded scalar function, by~\cref{lem:pettis_integrability_criteria}, it remains only to verify scalar measurability with respect to a total set in $F^*$. By the compatibility of dual pairs of $F$, $F^*$ and $\Phi$, $\Phi^*$, we have that
$$
\langle i(\phi) | K_\epsilon(\lambda_*-\lambda)u_\lambda \rangle = \langle K_\epsilon(\lambda_*-\lambda)u_\lambda | \phi \rangle,
$$
is $\mu_f$-integrable and, in particular, $\mu_f$-measurable. This establishes Pettis integrability.

Now, since $u_{\lambda_*}^{(\epsilon)}\in F$, by~\cref{lem:wp_expansion} and compatible dual pairs we have that for any $\phi\in\Phi$,
\begin{equation*}
\langle i(\phi) | u_{\lambda_*}^{(\epsilon)} \rangle = \langle u_{\lambda_*}^{(\epsilon)}|\phi \rangle = \int_{-\infty}^\infty K_\epsilon(\lambda-\lambda_*) \langle u_\lambda|\phi\rangle \,\dd\mu_f(\lambda) = \int_{-\infty}^\infty K_\epsilon(\lambda-\lambda_*) \langle i(\phi) | u_\lambda\rangle \,\dd\mu_f(\lambda).
\end{equation*}
Therefore, $u_{\lambda_*}^{(\epsilon)}$ is the Pettis integral of $\lambda\rightarrow K_\epsilon(\lambda_*-\lambda)u_\lambda$ (over $\mathbb{R}$) in $F$. To conclude the proof, we apply~\cref{lem:borel_meas_approx} directly to the Borel measures defined by
$$
\omega_\psi(\Omega) = \int_\Omega K_\epsilon(\lambda_*-\lambda)\langle \psi | u_\lambda \rangle\,\dd\mu_f(\lambda).
$$
For a $\mu_f$-integrable continuous seminorm on $F$, we simply note that Pettis integrability implies that the seminorm commutes with the integral, so that~\cite{thomas1975integration}
$$
|u_\lambda^{(\epsilon)}|=\int_{-\infty}^\infty K_\epsilon(\lambda_*-\lambda)|u_\lambda|\,\dd\mu_f(\lambda),
$$
and we can apply~\cref{lem:borel_meas_approx} to the induced Borel measure.
\end{proof}

\subsection*{Acknowledgements}

We would like to thank Semyon Dyatlov, Zhenhao Li, and Maciej Zworski for interesting discussions during the completion of this work and introducing us to the internal waves problem.

\bibliographystyle{abbrv}
\bibliography{GE_bib}

\end{document}